\documentclass[reqno,twoside,11pt]{amsart}

\usepackage{amsthm, amssymb, amsmath, amsfonts,eufrak}
\usepackage{verbatim}
\usepackage{mathptmx}
\usepackage{mathrsfs}
\usepackage{hyperref}
\usepackage{graphicx}

\DeclareFontFamily{OT1}{eusb}{} \DeclareFontShape{OT1}{eusb}{m}{n} {<5> <6> <7> <8> <9> <10> <11> <12> <14.4> eusb10}{}
\DeclareMathAlphabet{\eusb}{OT1}{eusb}{m}{n}

\DeclareFontFamily{OT1}{eusm}{} \DeclareFontShape{OT1}{eusm}{m}{n} {<5> <6> <7> <8> <9> <10> <11> <12> <14.4> eusm10}{}
\DeclareMathAlphabet{\eusm}{OT1}{eusm}{m}{n}

\DeclareFontFamily{OT1}{eufm}{} \DeclareFontShape{OT1}{eufm}{m}{n} {<5> <6> <7> <8> <9> <10> <11> <12> <14.4> eufm10}{}
\DeclareMathAlphabet{\mathfrak}{OT1}{eufm}{m}{n}

\DeclareFontFamily{OT1}{fraktura}{}
\DeclareFontShape{OT1}{fraktura}{m}{n} {<5> <6> <7> <8> <9> <10> <11> <12> <13> <14.4> [1.1] eufm10}{}
\DeclareMathAlphabet{\fraktura}{OT1}{fraktura}{m}{n}

\DeclareFontFamily{OT1}{cmfi}{} \DeclareFontShape{OT1}{cmfi}{m}{n} {<5> <6> <7> <8> <9> <10> <11> <12> <13> <14.4> [0.9] cmfi10}{}
\DeclareMathAlphabet{\cmfi}{OT1}{cmfi}{b}{n}

\DeclareFontFamily{OT1}{cmss}{} \DeclareFontShape{OT1}{cmss}{m}{n} {<5> <6> <7> <8> <9> <10> <11> <12> <13> <14.4> cmss10}{}
\DeclareMathAlphabet{\cmss}{OT1}{cmss}{m}{n}

\newtheoremstyle{thm}{1.8ex}{1.8ex}{\itshape\rmfamily}{} {\bfseries\rmfamily}{}{2ex}{}

\newtheoremstyle{def}{1.8ex}{1.8ex}{\slshape\rmfamily}{} {\bfseries\rmfamily}{}{2ex}{}

\newtheoremstyle{rem}{1.8ex}{1.8ex}{\rmfamily}{} {\bfseries\rmfamily}{}{2ex}{}

\newenvironment{proofsect}[1] {\vspace{0.2cm}\noindent{\rmfamily\itshape#1.}}{\qed\vspace{0.15cm}}

\theoremstyle{thm}
\newtheorem{theorem}{Theorem}[section]
\newtheorem{lemma}[theorem]{Lemma}
\newtheorem{proposition}[theorem]{Proposition}

\newtheorem*{Main Theorem}{Main Theorem.}
\newtheorem{corollary}[theorem]{Corollary}
\newtheorem*{special theorem}{Lindeberg-Feller Theorem for Martingales}

\theoremstyle{def}

\theoremstyle{rem}
\newtheorem{remark}[theorem]{Remark}
\newtheorem{remarks}[theorem]{Remarks}

\numberwithin{equation}{section}

\renewcommand{\small}{\fontsize{9}{9}\selectfont}

\renewcommand{\section}{\secdef\sct\sect}
\newcommand{\sct}[2][default]{%
\refstepcounter{section}
\addcontentsline{toc}{section}{{\tocsection {}{\thesection}{\!\!\!\!#1\dotfill}}{}}
\vspace{0.7cm}
\centerline{\scshape\thesection.\ #1} \nopagebreak \vspace{0.2cm}}
\newcommand{\sect}[1]{%
\vspace{0.4cm} \centerline{\large\scshape\rmfamily #1}
\vspace{0.2cm}}

\renewcommand{\subsection}{\secdef\subsct\sbsect}
\newcommand{\subsct}[2][default]{\refstepcounter{subsection}
\addcontentsline{toc}{subsection}
{{\tocsection{\!\!}{\hspace{1.2em}\thesubsection}{\!\!\!\!#1\dotfill}}{}}
\nopagebreak\vspace{0.45\baselineskip} {\flushleft\bf
\thesubsection~\bf #1.~}
\\*[3mm]\noindent
\nopagebreak}
\newcommand{\sbsect}[1]{\vspace{0.1cm}\noindent
\textbf{#1.~}\vspace{0.1cm}}

\renewcommand{\subsubsection}{%
\secdef \subsubsect\sbsbsect}
\newcommand{\subsubsect}[2][default]{%
\refstepcounter{subsubsection} 
\addcontentsline{toc}{subsubsection}{{\tocsection{\!\!}
{\hspace{3.05em}\thesubsubsection}{\!\!\!\!#1\dotfill}}{}}
\nopagebreak
\vspace{0.15\baselineskip} \nopagebreak {\flushleft\rmfamily
\itshape\thesubsubsection
\ \rmfamily #1\/.}\ }
\newcommand{\sbsbsect}[1]{\vspace{0.1cm}\noindent
\rmfamily \itshape
\arabic{section}.\arabic{subsection}.\arabic{subsubsection} \
\sffamily #1\/.\ }

\renewcommand{\caption}[1]{%
\vglue0.5cm
\refstepcounter{figure}
\begin{minipage}{0.9\textwidth}\small {\sc Figure~\thefigure. }#1\end{minipage}}

\def\myffrac#1#2 in #3{\raise 2.6pt\hbox{$#3 #1$}\mkern-1.5mu\raise 0.8pt\hbox{$#3/$}\mkern-1.1mu\lower 1.5pt\hbox{$#3 #2$}}
\newcommand{\ffrac}[2]{\mathchoice%
{\myffrac{#1}{#2} in \scriptstyle}
{\myffrac{#1}{#2} in \scriptstyle}
{\myffrac{#1}{#2} in \scriptscriptstyle}
{\myffrac{#1}{#2} in \scriptscriptstyle}
}

\newcommand{\dist}{\operatorname{dist}}
\newcommand{\supp}{\operatorname{supp}}

\newcommand{\textd}{\text{\rm d}\mkern0.5mu}

\newcommand{\1}{\operatorname{\sf 1}\!}

\newcommand{\BB}{\mathcal B}

\newcommand{\FF}{\mathcal F}

\newcommand{\HH}{\mathcal H}

\newcommand{\KK}{\mathcal K}
\newcommand{\LL}{\mathcal L}

\newcommand{\NN}{\mathcal N}

\newcommand{\B}{\mathbb B}

\newcommand{\E}{\mathbb E}

\newcommand{\N}{\mathbb N}

\newcommand{\BbbP}{\mathbb P}

\newcommand{\R}{\mathbb R}

\newcommand{\Z}{\mathbb Z}

\newcommand{\scrF}{\mathscr{F}}

\newcommand{\twoeqref}[2]{(\ref{#1}--\ref{#2})}
\newcommand{\cc}{{\text{\rm c}}}

\newcommand{\hate}{\hat{\text{\rm e}}}

\newcommand{\eff}{{\text{\rm eff}}}
\newcommand{\Var}{\text{\rm Var}}

\renewcommand{\LL}{\cmss L}
\renewcommand{\gg}{\mathfrak g}

\title[A CLT for effective conductance]
{\large A central limit theorem for the effective conductance: Linear boundary data and small ellipticity contrasts}
\author[M.~Biskup, M.~Salvi, T.~Wolff]
{M.~Biskup$^{1,2}$,\,\,\,\,M.~Salvi$^{3}$ \,\,and\,\, T.~Wolff$^{3,4}$}

\begin{document} 
\thanks{\small\hglue-4.5mm
\copyright\,\textrm{2014} \textrm{M.~Biskup, M.~Salvi and T.~Wolff}.
Reproduction, by any means, of the entire article for non-commercial purposes is permitted without charge.}

\maketitle
\vspace{-5mm}
\centerline{\textit{$^1$Department of Mathematics, UCLA, Los Angeles, California, U.S.A.}}
\centerline{\textit{$^2$School of Economics, University of South Bohemia, \v Cesk\'e Bud\v ejovice, Czech Republic}}
\centerline{\textit{$^3$Institut f\"ur Mathematik, Technische Universit\"at Berlin, Berlin, Germany}}
\centerline{\textit{$^4$Weierstra\ss-Institut f\"ur Angewandte Analysis und Stochastik, Berlin, Germany}
}

\vspace{-3mm}
\begin{abstract}
Given a resistor network on~$\Z^d$ with nearest-neighbor conductances, the effective conductance in a finite set with a given boundary condition is the the minimum of the Dirichlet energy over functions with the prescribed boundary values. For shift-ergodic conductances, linear (Dirichlet) boundary conditions and square boxes, the effective conductance scaled by the volume of the box converges to a deterministic limit as the box-size tends to infinity. Here we prove that, for i.i.d.\ conductances with a small ellipticity contrast, also a (non-degenerate) central limit theorem holds. The proof is based on the corrector method and the Martingale Central Limit Theorem; a key integrability condition is furnished by the Meyers estimate. More general domains, boundary conditions and  ellipticity contrasts will be addressed in a subsequent paper.
\end{abstract}


\vglue0.3cm


\section{Introduction and Main Result}
\label{sec1}
\noindent
As is well known, most materials, regardless how pure they may seem at the macroscopic level, have a rather complicated  microscopic structure. It may then come as a surprise that physical phenomena such as heat or electric conduction are described so well using differential equations with smooth, sometimes even constant, coefficients. An explanation has been offered by homogenization theory: rapid oscillations at the microscopic level average out, or homogenize, at the macroscopic scale. However, this does not mean that the microscopic structure is simply washed out. Indeed, while it disappears from the structure of the resulting equations, it remains embedded in the values of  effective material constants, e.g., the coefficients.

An illustrative example of a homogenization problem is that of effective conductance. We will formulate an instance of this problem in the setting of resistor networks. Let~$\Z^d$ denote the $d$-dimensional hypercubic lattice and suppose that each unordered nearest-neighbor edge~$\langle x,y\rangle$ is assigned a value $a_{xy}=a_{yx}\in(0,\infty)$ --- called the \emph{conductance} of $\langle x,y\rangle$. For any $\Lambda\subset\Z^d$, let~$\B(\Lambda)$ be the edges with at least one endpoint in $\Lambda$. Given an $f\colon\Z^d\to\R$ and a finite $\Lambda\subset\Z^d$, let
\begin{equation}
\label{E:1.1}
Q_\Lambda(f):=\sum_{\langle x,y\rangle\in\B(\Lambda)}a_{xy}\bigl[f(y)-f(x)\bigr]^2,
\end{equation}
where each pair $(x,y)$ is counted only once. This is the electrostatic (Dirichlet) energy for the potential~$f$ with Dirichlet boundary condition on the boundary vertices of~$\Lambda$. 

Consider now the square box $\Lambda_L:=[0,L)^d\cap\Z^d$. A quantity of prime interest for us is the \emph{effective conductance},
\begin{equation}
\label{E:1.2}
C^{\,\eff}_{L}(t):=\inf\bigl\{Q_{\Lambda_L}(f)\colon f(x)=t\cdot x,\,\,\forall x\in\partial\Lambda_L\bigr\},
\end{equation}
where $t\in\R^d$ and where $\partial\Lambda$ are those vertices outside $\Lambda$ that have an edge into~$\Lambda$. By Kirchhoff's and Ohm's laws (see, e.g., Doyle and Snell~\cite{Doyle-Snell}), $C^{\,\eff}_{L}(t)$ is the total electric current flowing through the network when the boundary vertices are kept at voltage~$t\cdot x$.

For homogeneous resistor networks, i.e., when $a_{xy}:=a$ for all $\langle x,y\rangle$, the infimum \eqref{E:1.2} is achieved by $f(x):=t\cdot x$ and so $C^{\,\eff}_L(t)=a|t|^2L^d(1+o(1))$. A question of (reasonably) practical interest is then what happens when the conductances~$a_{xy}$ are no longer constant, but remain close to a constant. In particular, we may assume that they are uniformly elliptic, i.e., 
\begin{equation}
\label{E:1.3}
\exists\lambda\in(0,1),\,\,\forall\langle x,y\rangle\in\B(\Z^d)\colon\qquad \lambda\le a_{xy}\le\frac1\lambda.
\end{equation}
A comparison of $Q_\Lambda$ with these $a_{xy}$'s and the homogeneous case shows that $C^{\,\eff}_L(t)$ is still of the order of $|t|^2L^d$. Moreover, thanks to the choice of the linear boundary condition, by subadditivity arguments the limit 
\begin{equation}
c_\eff(t):=\lim_{L\to\infty}\,\, \frac1{L^d}\,C^{\,\eff}_L(t)
\end{equation}
exists almost surely for any ergodic distribution of the conductances. The problem left to resolve is thus a computation of the limit value. 

Although $c_\eff(t)$ can be computed explicitly only in a handful of (mostly periodic) cases, it can be characterized in large generality: Suppose that $a_{xy}=a_{xy}(\omega)$ is a sample from a shift-ergodic law~$\BbbP$ on the product space $\Omega:=\bigotimes_{\B(\Z^d)}[\lambda,\ffrac1\lambda]$ indexed by edges of~$\Z^d$. (Technically, $a_{xy}(\omega)=a_{yx}(\omega)$ is the coordinate projection on edge $\langle x,y\rangle$.) 
As is well known (Papanicolaou and Varadhan~\cite{Papanicolaou-Varadhan}, Kozlov~\cite{Kozlov} and K\"unnermann~\cite{Kuennermann} and the book Jikov, Kozlov and Oleinik~\cite{JKO}),
\begin{equation}
\label{E:1.5}
c_\eff(t) = \inf_{g\in L^\infty(\BbbP)}\, \E\biggl(\,\,\sum_{x=\hate_1,\dots,\hate_d} a_{0,x}(\omega)\bigl|t\cdot x+\nabla_xg(\omega)\bigr|^2\biggr).
\end{equation}
Here $\E$ is expectation with respect to~$\BbbP$, $\hate_1,\dots,\hate_d$ are the unit coordinate vectors in~$\R^d$ and $\nabla_x g(\omega):=g\circ\tau_x(\omega)-g(\omega)$ is the gradient of~$g$ in direction of~$x\in\Z^d$ with $\tau_x$ denoting the shift by~$x$; i.e., the map on $\Omega$ such that~$a_{yz}(\tau_x\omega):=a_{x+y,x+z}(\omega)$. The expression in \eqref{E:1.5} can be interpreted as the  Dirichlet energy density --- with the spatial average naturally replaced by the ensemble average.

\smallskip
Once the (deterministic) leading-order of $C^{\,\eff}_L(t)$ has been identified, the next natural question is that of fluctuations. It is obvious --- e.g., by checking the explicitly computable $d=1$ case --- that no universal limit law can be expected for general conductance distributions, but progress could perhaps be made for the (physically most appealing) case of i.i.d.~conductances. However, even here establishing just the order of magnitude of the fluctuations turned out to be an arduous task. Indeed, more than a decade ago Wehr~\cite{Wehr} showed that $\Var(C^{\,\eff}_L)\ge c L^{d}$ for some~$c>0$ but a corresponding upper bound has been furnished only recently by Gloria and Otto~\cite{Gloria-Otto}. Both of these results contain important technical caveats: Wehr requires continuously distributed $a_{xy}$'s while Gloria and Otto express their results under a ``massive'' cutoff (at least in dimension 2). 

Gloria and Otto drew important ideas from an earlier unpublished note by Naddaf and Spen\-cer~\cite{Naddaf-Spencer} where (optimal) upper bounds on the variance are derived for certain correlated conductance laws. The main tool of \cite{Naddaf-Spencer} is the \emph{Meyers estimate} (cf Meyers~\cite{Meyers}), to be used heavily in the present note as well. Other noteworthy earlier derivations of (suboptimal) variance upper bounds include an old paper by Yurinskii~\cite{Yurinskii} and a more recent paper by Benjamini and Rossignol~\cite{Benjamini-Rossignol}. Closely related to these estimates are recent derivations of quantitative central limit theorems for random walk among random conductances and approximations of the limiting diffusivity matrix, e.g., Caputo and Ioffe~\cite{Caputo-Ioffe}, Bourgeat and Piatnitski~\cite{Bourgeat-Piatnitski}, Boivin~\cite{Boivin}, Mourrat~\cite{Mourrat}, Gloria and Mourrat~\cite{Gloria-Mourrat1,Gloria-Mourrat2}, etc.
Incidentally, the Meyers estimate is also the key tool in~\cite{Caputo-Ioffe}.

\smallskip
The goal of the present note is to prove that, for i.i.d.\ conductances which are (deterministically) not too far from a constant, the asymptotic law of $C^{\,\eff}_L(t)$ is in fact Gaussian. Let $\NN(\mu,\sigma^2)$ denote the normal random variable with mean~$\mu$ and variance~$\sigma^2$. Then we have: 

\begin{theorem}
\label{thm1}
Suppose the conductances $a_{xy}$ are i.i.d. For each $d\ge1$, there is $\lambda=\lambda(d)\in(0,1)$ such that the following holds: If \eqref{E:1.3} is satisfied $\BbbP$-a.s.\ with this~$\lambda$, then for each $t\in\R^d$ there is $\sigma_t^2\in[0,\infty)$ such that
\begin{equation}
\label{E:2.1}
\frac{C_L^{\,\eff}(t)-\E C_L^{\,\eff}(t)}{|\Lambda_L|^{1/2}}\,\,\,\overset{\text{\rm law}}{\underset{L\to\infty}\longrightarrow}\,\,\,\NN(0,\sigma_t^2).
\end{equation}
Whenever the conductance law is non-degenerate we have $\sigma_t^2>0$ for all $t\ne0$.
\end{theorem}

The proof also immediately yields:

\begin{corollary}
\label{cor1.2}
Under the conditions of Theorem~\ref{thm1}, 
\begin{equation}
\frac1{|\Lambda_L|}\text{\rm Var}\bigl(C_L^{\,\eff}(t)\bigr)
\,\,\underset{L\to\infty}\longrightarrow\,\,\sigma_t^2,
\end{equation}
where $\sigma_t^2$ is as in \eqref{E:2.1}.
\end{corollary}

A few remarks are in order:

\begin{remarks}
(1) Notice that \eqref{E:2.1} does not give us much information on the ``order expansion'' of $C_L^{\,\eff}(t)$. Indeed, we know that $\E C_L^{\,\eff}(t)$ is to the leading order equal to $c_\eff(t)|\Lambda_L|$ but when this order is subtracted, the next-order term is (presumably) of boundary size. In $d\ge3$, this is still larger than the typical size of the fluctuations. (This is referred to as the systematic error in \cite{Gloria-Otto}.) Notwithstanding, what \eqref{E:2.1} does tell us is the character of the leading order \emph{random} term.

(2) There is in fact a formula for $\sigma_t^2$, see Theorem~\ref{thm2} below, which also shows that $t\mapsto\sigma_t^2$ is a  bi-quadratic (and thus smooth) function of~$t$. However, the formula involves complicated conditioning and does not seem very useful for practical computations.

(3) There is no restriction on the single-conductance law other than \eqref{E:1.3}. In particular, this law can have a non-absolutely continuous part including atoms. Certain technical problems do arise at this level of generality; see Section~\ref{sec2.5} which, we believe, is of independent interest.
\end{remarks}

We prove Theorem~\ref{thm1} by reducing it to the Martingale Central Limit Theorem. There are two main technical ingredients: homogenization theory (which enables a stationary {martingale approximation} of~$C^{\,\eff}_L(t)$) and analytical estimates for finite-volume harmonic coordinates (by which we control the errors in the martingale approximation). The restrictions to rectangular boxes, linear boundary conditions and small ellipticity contrasts permit us to encapsulate the analytical input into a single step, the {Meyers estimate}, cf Proposition~\ref{prop:Meyers estimate} and~Theorem~\ref{Lpbound-K}. These restrictions can be relaxed but not without  additional arguments not all of which have been handled satisfactorily at this time. These are deferred to a follow-up paper. 

\smallskip
We remark that two recent preprints have been brought to our attention at the time this work was first announced in conference talks. First, Nolen~\cite{Nolen} has established a normal approximation to the effective conductance defined over a periodic environment, in the limit when the period tends to infinity. Second, in a preprint that was posted at the time of writing the present note, Rossignol~\cite{Rossignol} formulates and proves a central limit law for the \emph{effective resistance} for the corresponding problem on a torus. Nolen defines the problem over continuum, albeit with a rather strong assumption on an underlying Gaussian i.i.d.\ structure. Rossignol's setting is based on minimizing the electrostatic energy over {currents} (rather than potentials) subject to a restriction on the total current flowing around the torus. By a well known reciprocity relation between effective conductance and resistance, these papers appear to address similar problems.

The present paper differs from both Nolen~\cite{Nolen} and Rossignol~\cite{Rossignol} primarily in its emphasis on fixed (Dirichlet), as opposed to periodic, boundary conditions. Indeed, a majority of our technical work is aimed at controlling the resulting boundary effects. Also the way a Gaussian limit law is established is quite different: Nolen appeals to Stein's method, Rossignol uses noise-sensitivity tools while we invoke the Martingale Central Limit Theorem. A notable deficiency of the present paper compared to~\cite{Nolen} and~\cite{Rossignol} is the limitation on ellipticity contrast. Nolen overcomes this by an appeal to Gloria and Otto~\cite{Gloria-Otto}, although this ultimately precludes the most interesting conclusion (namely, the CLT without ``massive'' cutoff) in~$d=2$. Rossignol's approach appears to work seamlessly for all elliptic product laws thanks to ``enhanced'' form of averaging coming from the periodic boundary condition. 

While the Gloria-Otto method can perhaps be adapted to our situation as well, just as for Nolen~\cite{Nolen} it fails to deliver the desired conclusion in $d=2$. The issue is that the method yields bounds on the moments of the corrector, which diverge in $d=2$, while we need only moments of the gradients of the corrector. (\textit{Update in revised version}: This issue has now been overcome in Gloria, Neukamm and Otto  \cite{GNO}, albeit only in either infinite volume or for periodic boundary conditions.) Notwithstanding, our point of view is that the moment bounds seem to be a separate technical matter, and so, for the present paper, we decided to sacrifice on generality of the distribution and derived the CLT only in the simplest, albeit still physically interesting, case.

\section{Key ingredients}
\label{sec2}\noindent
Here we discuss the strategy of the proof of Theorem~\ref{thm1} and state its principal ingredients in the form of suitable propositions. The actual proofs begin in Section~\ref{sec3}.

\vglue-1mm

\subsection{Martingale approximation}
 A standard way to control fluctuations of a function of i.i.d.\ random variables is by way of a \emph{martingale approximation}. Let us order the random variables $\{a_{xy}\colon \langle x,y\rangle\in\B(\Lambda_L)\}$ in any (for now) convenient way and let $\FF_k$ to be the $\sigma$-algebra generated by the first~$k$ of them. (Since we only aim at a distributional convergence, the $\sigma$-algebras may depend on~$L$.) Then
\begin{equation}
\label{E:2.1q}
C_L^{\,\eff}(t)-\E C_L^{\,\eff}(t)
=\sum_{k=1}^{|\B(\Lambda_L)|}
Z_k,
\end{equation}
where
\begin{equation}
\label{E:2.3}
Z_k:=\E\bigl(C_L^{\,\eff}(t)\big|\scrF_k\bigr)-\E\bigl(C_L^{\,\eff}(t)\big|\scrF_{k-1}\bigr).
\end{equation}
Obviously, the quantity $Z_k$ is a martingale increment. In order to show distributional convergence to $\NN(0,\sigma^2)$, it suffices to verify the (Lindenberg-Feller-type) conditions of the Martingale Central Limit Theorem due to Brown~\cite{Brown}:
\begin{enumerate}
\item[(1)] 
There exists $\sigma^2\in[0,\infty)$ such that
\begin{equation}
\label{E:2.4a}
\frac1{|\Lambda_L|}\sum_{k=1}^{|\B(\Lambda_L)|}
\E(Z_k^2|\FF_{k-1})\,\underset{L\to\infty}\longrightarrow\,\sigma^2
\end{equation}
in probability, and
\item[(2)]
for each $\epsilon>0$,
\begin{equation}
\label{E:2.5a}
\frac1{|\Lambda_L|}\sum_{k=1}^{|\B(\Lambda_L)|}
\E\bigl(\,Z_k^2\1_{\{|Z_k|>\epsilon|\Lambda_L|^{1/2}\}}\big|\FF_{k-1}\bigr)\,\underset{L\to\infty}\longrightarrow\,0
\end{equation}
in probability.
\end{enumerate}
The sums on the left suggest invoking the Spatial Ergodic Theorem, but for that we would need to ensure that the individual terms in the sum are (at least approximated by) functions that are stationary with respect to shifts of~$\Z^d$. This necessitates the following additional input:
\begin{enumerate}
\item[(i)] a specific choice of the ordering of the edges, and
\item[(ii)] a more explicit representation for $Z_k$.
\end{enumerate}
We will now discuss various aspects of these in more detail.

\subsection{Stationary edge ordering}
\label{sec2.2}\noindent
Recall that $\B(\Z^d)$ denotes the set of all (unordered) edges in~$\Z^d$. We will order $\B(\Z^d)$ as follows: Let $\preceq$ denote the lexicographic ordering of the vertices of~$\Z^d$. Explicitly, for $x=(x_1,\dots,x_d)$ and $y=(y_1,\dots,y_d)$ we have $x\preceq y$ if either $x=y$ or $x\ne y$ and there exists $i\in\{1,\dots,d\}$ such that $x_j=y_j$ for all $j<i$ and $x_i<y_i$. We will write $x\prec y$ if $x\ne y$ and $x\preceq y$. 

For the purpose of defining a stationary ordering of the edges, and also easier notation in some calculations that are to follow, we now identify $\B(\Z^d)$ with the set of pairs $(x,i)$, where $x\in\Z^d$ and $i\in\{1,\dots,d\}$, so that $(x,i)$ corresponds to the edge between the vertices $x$ and $x+\hate_i$. We will then write
\begin{equation}
(x,i)\preceq (y,j)
\quad\text{if}\quad\left\{\begin{aligned}
&\text{either } x\prec y
\\
&\text{or } 
\quad x=y \text{ and } i\le j.
\end{aligned}\right.
\end{equation}
Again, $(x,i)\prec(y,j)$ if $(x,i)\preceq(y,j)$ but $(x,i)\ne(y,j)$. It is easy to check that $\preceq$ is a complete order on~$\B(\Z^d)$. A key fact about this ordering is its stationarity with respect to shifts:

\begin{lemma}
\label{lemma-2.3}
If $(x,i)\preceq(y,j)$ then also $(x+z,i)\preceq(y+z,j)$ for all $z\in\Z^d$.
\end{lemma}

\begin{proofsect}{Proof}
This is a trivial consequence of the definition.
\end{proofsect}

Now we proceed to identify the sigma algebras $\{\scrF_k\}$ in the martingale representation above. Recall that $\Omega:=\bigotimes_{\B(\Z^d)}[\lambda,\ffrac1\lambda]$ denotes the set of conductance configurations satisfying \eqref{E:1.3}. Writing~$\omega$ for elements of $\Omega$ we use $a_{xy}=a_{xy}(\omega)$, for $\langle x,y\rangle\in\B(\Z^d)$, to denote the coordinate projection corresponding to edge~$\langle x,y\rangle$. Given $L\ge1$, set $N:=|\B(\Lambda_L)|$ and let $b_1,\dots,b_N$ be the enumeration of $\B(\Lambda_L)$ induced by the ordering of edges $\preceq$ defined above. Then we set
\begin{equation}
\scrF_k:=\sigma(\omega_b\colon b\preceq b_k),\qquad k=1,\dots,N,
\end{equation}
with
\begin{equation}
\scrF_0:=\sigma(\omega_b\colon b\prec b_1).
\end{equation}
By definition $\scrF_0$ is independent of the edges in $\B(\Lambda_L)$ while $\scrF_N$ determines the entire configuration in $\B(\Lambda_L)$. Note also that $\scrF_k$ includes information about edges that are not in $\B(\Lambda_L)$. This will be of importance once we replace $Z_k$ by a random variable that depends on all of~$\omega$.


\begin{figure}[t]
\centerline{\includegraphics[width=2.8truein]{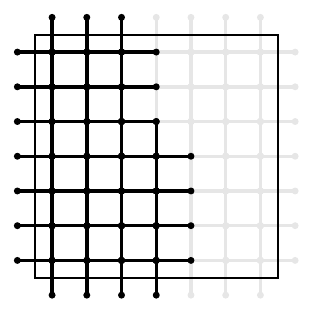}}
\bigskip
\begin{quote}
\small
{\sc Figure~1.}\ The set of edges (drawn in bold) in~$\B(\Lambda_L)$ that determine the events in~$\scrF_k$. Here $d:=2$,~$L:=7$ and $k:=61$. In our ordering of edges on~$\Z^2$, this corresponds to the edges up to and including $(x,i)$ for $x:=(3,3)$ and $i:=2$.  
\label{fig1}
\normalsize
\end{quote}
\end{figure}

\subsection{An explicit form of martingale increment}
\label{sec2.3}\noindent
Having addressed the ordering of the edges, and thus the definition of the $\sigma$-algebras $\scrF_k$, we now proceed to derive a more explicit form of the quantity $Z_k$ from \eqref{E:2.3}. Given $\omega\in\Omega$, define the operator $\LL_\omega$ on ($\R$ or~$\R^d$-valued) functions on the lattice via
\begin{equation}
\label{E:2.8a}
(\LL_\omega f)(x):=\sum_{y\colon\langle x,y\rangle\in\B(\Z^d)}a_{xy}(\omega)\,\bigl[f(y)-f(x)\bigr].
\end{equation}
This is an elliptic finite-difference operator --- a random Laplacian --- that arises as the generator of the random walk among random conductances $\{a_{xy}(\omega)\}$ (see, e.g.,~Biskup~\cite{Biskup-review} for a review of these connections). The existence/uniqueness for the associated Dirichlet problem implies that for any finite $\Lambda\subset\Z^d$ there is a unique $\Psi_\Lambda\colon\Omega\times(\Lambda\cup\partial\Lambda)\to\R^d$ such that $x\mapsto\Psi_\Lambda(\omega,x)$ obeys
\begin{equation}
\label{eqn-harmonic-coordinate-finite}
\left\{\begin{alignedat}{3}
&\LL_\omega\,\Psi_\Lambda(\omega,x)=0, \quad&\qquad&x\in\Lambda,
\\*[1mm]
&\Psi_\Lambda(\omega,x)=x, &\qquad &x\in\partial\Lambda.
\end{alignedat}\right.
\end{equation}
It is then easily checked that $f(x):=t\cdot\Psi_\Lambda(\omega,x)$ is the unique minimizer of $f\mapsto Q_\Lambda(f)$ over all functions~$f$ with the boundary values $f(x)=t\cdot x$ for $x\in\partial\Lambda$. In particular, we have
\begin{equation}
C_L^{\,\eff}(t)=Q_{\Lambda_L}\bigl(\,t\cdot\Psi_{\Lambda_L}\bigr)
\end{equation}
for all $t\in\R^d$. The function $x\mapsto\Psi_\Lambda(\omega,x)$ will sometimes be referred to as a \emph{finite-volume harmonic coordinate}. (The first line in \eqref{eqn-harmonic-coordinate-finite} justifies this term.)

The minimum value $Q_\Lambda(t\cdot\Psi_\Lambda)$ is a non-decreasing, continuous and concave function of $\{a_{xy}\colon \langle x,y\rangle\in\B(\Lambda)\}$. Thanks to the uniqueness of the solution to \eqref{eqn-harmonic-coordinate-finite},  $Q_\Lambda(t\cdot\Psi_\Lambda)$ is also continuously differentiable in $a_{xy}$'s and it is easy to see, keeping in mind again that $t\cdot\Psi_\Lambda$ is the minimizer and using the Euler-Lagrange equation for the minimization problem, that
\begin{equation}
\label{E:2.4}
\frac{\partial}{\partial a_{xy}}Q_\Lambda(t\cdot\Psi_\Lambda)=\bigl[t\cdot\Psi_\Lambda(\omega,y)-t\cdot\Psi_\Lambda(\omega,x)\bigr]^2,\qquad \langle x,y\rangle\in\B(\Lambda).
\end{equation}
This relation is of fundamental importance for what is to come.

Abusing the notation slightly, let $\omega_1,\dots,\omega_N$, with $N:=|\B(\Lambda)|$, denote the components of the configuration $\omega$ over $\B(\Lambda)$ labeled in the order induced by $\preceq$ defined above. Let
\begin{equation}
q(\omega_1,\dots,\omega_N):=Q_\Lambda(t\cdot\Psi_\Lambda)
\end{equation}
mark explicitly the dependence of the right-hand side on these variables. The product structure of the underlying probability measure then allows us to give a more explicit expression for the increment $Z_k=Z_k(\omega_1,\dots,\omega_k)$:
\begin{equation}
\label{E:2.10}
\begin{aligned}
Z_k&=\int \BbbP(\textd\omega_k')\dots\BbbP(\textd\omega_N')\bigl[q(\omega_1,\dots,\omega_k,\omega_{k+1}',\dots,\omega_N')
\\
&\qquad\qquad\qquad\qquad\qquad\qquad\quad-q(\omega_1,\dots,\omega_{k-1},\omega_k',\dots,\omega_N')\bigr]
\\
&=\int \BbbP(\textd\omega_k')\dots\BbbP(\textd\omega_N')\int_{\omega_k'}^{\omega_k}\textd\tilde\omega_k\,
\frac\partial{\partial\tilde\omega_k}q(\omega_1,\dots,\omega_{k-1},\tilde\omega_k,\omega_{k+1}',\dots,\omega_N'),
\end{aligned}
\end{equation}
with the inner integral in Riemann sense. A key point is that the last partial derivative is (modulo notational changes) given by \eqref{E:2.4}, i.e., $Z_k$ is the modulus-squared of the gradient of $t\cdot\Psi_\Lambda$ over the $k$-th edge in $\B(\Lambda)$ integrated over part of the variables. To see that $Z_k$ is a martingale increment note that the Riemann integral changes sign when its limits are interchanged.


\subsection{Input from homogenization theory}
In order to apply the Spatial Ergodic Theorem to the sums on the left of \twoeqref{E:2.4a}{E:2.5a}, we will substitute for~$Z_k$ a quantity that is stationary with respect to the shifts of~$\Z^d$. This will be achieved by replacing the discrete gradient of $\Psi_\Lambda$ --- which by \eqref{E:2.4} enters as the partial derivative of~$q$ in the formula for~$Z_k$ --- by the gradient of its infinite-volume counterpart, to be denoted by $\psi$. The existence and properties of the latter object are standard:

\begin{proposition}[Infinite-volume harmonic coordinate]
\label{prop-corrector}
Suppose the law of the conductances is ergodic with respect to the shifts of~$\Z^d$ and assume \eqref{E:1.3} for some $\lambda\in(0,1)$. Then there is a function $\psi\colon\Omega\times\Z^d\to\R^d$ such that
\settowidth{\leftmargini}{(11,)}
\begin{enumerate}
\item[(1)] ($\psi$ is $\LL_\omega$-harmonic)
$\LL_\omega\psi(\omega,x)=0$ for all $x$ and $\BbbP$-a.e.~$\omega$.
\item[(2)] ($\psi$ is shift covariant) For $\BbbP$-a.e.~$\omega$ we have $\psi(\omega,0):=0$ and
\begin{equation}
\psi(\omega,y)-\psi(\omega,x)=\psi(\tau_x\omega,y-x),\qquad x,y\in\Z^d.
\end{equation}
\item[(3)] ($\psi$ is square integrable)
\begin{equation}
\E\biggl(\,\,\sum_{x=\hate_1,\dots,\hate_d}a_{0,x}(\omega)\bigl|\psi(\omega,x)\bigr|^2\biggr)<\infty.
\end{equation}
\item[(4)] ($\psi$ is approximately linear) The corrector
$\chi(\omega,x):=\psi(\omega,x)-x$ satisfies
\begin{equation}
\label{E:2.16}
\lim_{|x|\to\infty}\frac{\E\bigl(|\chi(\omega,x)|^2\bigr)}{|x|^2}=0.
\end{equation}
\end{enumerate}
\end{proposition}

\begin{proofsect}{Proof}
Properties (1-3) are standard and follow directly from the construction of~$\psi$ (which is done, essentially, by showing that a minimizing sequence in \eqref{E:1.5} converges in a suitable $L^2$-sense; see, e.g., Biskup~\cite[Section~3.2]{Biskup-review} for a recent account of this). As to (4), a moment's thought reveals that it suffices to show this for~$x$ of the form $n\hate_i$, where $n\to\pm\infty$. This follows from the Mean Ergodic Theorem, similarly as in \cite[Lemma~4.8]{Biskup-review}.
\end{proofsect}

The replacement of (the gradients of) $\Psi_\Lambda$ by $\psi$ necessitates developing means to quantify the resulting error. For this we introduce an $L^p$-norm on functions $f\colon\Omega\times(\Lambda\cup\partial\Lambda)\to\R^d$ by the usual formula
\begin{equation}
\Vert\nabla f\Vert_{\Lambda,p}:=\biggl(\,\frac1{|\Lambda|}\sum_{\langle x,y\rangle\in\B(\Lambda)}\E\bigl|f(\omega,y)-f(\omega,x)\bigr|^p\biggr)^{\ffrac 1p}.
\end{equation}
Analogously, we also introduce a norm on functions $\varphi\colon\Omega\times\Z^d\to\R^d$ by
\begin{equation}
\label{E:2.19}
\Vert\nabla\varphi\Vert_p:=\Bigl(\,\sum_{x=\hate_1,\dots,\hate_d}\E\bigl|\varphi(\omega,x)-\varphi(\omega,0)\bigr|^p\Bigr)^{\ffrac1p}.
\end{equation}
Here we introduced the symbol $\nabla f$ for an $\R^d$-valued function whose $i$-th component at~$x$ is given by $\nabla_if(x):=f(x+\hate_i)-f(x)$ --- abusing our earlier use of this notation.
It is reasonably well known, albeit perhaps not written down explicitly anywhere, that the gradients of $\Psi_\Lambda$ and~$\psi$ are close in $\Vert\cdot\Vert_{\Lambda,2}$-norm (see, however, Proposition~3.1 of Caputo and Ioffe~\cite{Caputo-Ioffe} for a torus version of this statement).

\begin{proposition}
\label{prop-approx-corrector-L2}
Suppose the law~$\BbbP$ on conductances $\{a_{xy}\}$ is ergodic with respect to shifts of~$\Z^d$ and obeys \eqref{E:1.3} for some $\lambda\in(0,1)$. Then
\begin{equation}
\bigr\Vert\nabla(\Psi_{\Lambda_L}-\psi)\bigr\Vert_{\Lambda_L,2}\,\,\underset{L\to\infty}\longrightarrow\,\,0.
\end{equation}
\end{proposition}

As we will elaborate on later (see Remark~\ref{R:representation}), this is exactly what is needed to establish the representation \eqref{E:1.5} for the limit value $c_\eff(t)$ of the sequence $L^{-d}C_L^{\,\eff}(t)$. However, in order to validate the conditions \twoeqref{E:2.4a}{E:2.5a} of the Martingale Central Limit Theorem, more than just square integrability is required. For this we state and prove:

\begin{proposition}[Meyers' estimate]
\label{prop:Meyers estimate}
Suppose~$\BbbP$ is ergodic with respect to shifts. For each $d\ge1$, there is $\lambda=\lambda(d)\in(0,1)$ such that if \eqref{E:1.3} holds $\BbbP$-a.s.\ with this~$\lambda$, then for some $p>4$,
\begin{equation}
\label{eqn-approx-corrector-Lp-1}
\Vert\nabla\psi\Vert_p<\infty
\end{equation}
and
\begin{equation}
\label{eqn-approx-corrector-Lp-2}
\sup_{L\ge1}\,\bigr\Vert\nabla(\Psi_{\Lambda_L}-\psi)\bigr\Vert_{\Lambda_L,p}<\infty.
\end{equation}
\end{proposition}

Proposition~\ref{prop:Meyers estimate} is the sole reason for our restriction on ellipticity contrast. We believe that, on the basis of the technology put forward in Gloria and Otto~\cite{Gloria-Otto}, no such restriction should be needed. To attest this we note that versions of the above bounds actually hold pointwise for a.e.~$\omega\in\Omega$ satisfying \eqref{E:1.3}; i.e., for norms without the expectation~$\E$. In addition, from \cite[Proposition~2.1]{Gloria-Otto} we in fact know \eqref{eqn-approx-corrector-Lp-1} for all $p\in(1,\infty)$ when $d\ge3$. \textit{Update in revised version}: We note that \eqref{eqn-approx-corrector-Lp-1} is now known in all $d\ge2$ from Gloria, Neukamm and Otto~\cite{GNO}. Unfortunately, this does not apply to \eqref{eqn-approx-corrector-Lp-2}.
A torus version of Proposition~\ref{prop:Meyers estimate} appeared in Theorem~4.1 of Caputo and Ioffe~\cite{Caputo-Ioffe}.

\subsection{Perturbed corrector and variance formula}
\label{sec2.5}\noindent
Unfortunately, a direct attempt at the substitution of (the gradients of) $\Psi_\Lambda$ by $\psi$ in \eqref{E:2.10} reveals another technical obstacle: As \eqref{E:2.10} relies on the Fundamental Theorem of Calculus, the replacement of $\Psi_\Lambda$ by~$\psi$ requires the latter function to be defined for~$\omega$ that may lie outside of the support of~$\BbbP$. This is a problem because $\psi$ is generally determined by conditions (1-4) in Proposition~\ref{prop-corrector} only on a set of full~$\BbbP$-measure. Imposing additional assumptions on~$\BbbP$ --- namely, that the single-conductance distribution is supported on an interval with a bounded and non-vanishing density --- would allow us to replace the Lebesgue integral in \eqref{E:2.10} by an integral with respect to $\BbbP(\textd\tilde\omega_k)$ and thus eliminate this problem. Notwithstanding, we can do much better by invoking a rank-one perturbation argument which we describe next. 

Fix an index $i\in\{1,\dots,d\}$ and recall the notation $\nabla_if(x):=f(x+\hate_i)-f(x)$. For a vertex $x\in\Z^d$ and a finite set $\Lambda\subset\Z^d$ satisfying $x\in\Lambda$ or $x+\hate_i\in\Lambda$, let  $\gg_\Lambda^{(i)}(\omega,x)$ be defined by 
\begin{equation}
\gg_{\Lambda}^{(i)}(\omega,x)^{-1}:=\inf\bigl\{Q_{\Lambda}(f)\colon 0\le f\le1,\,f(x+\hate_i)-f(x)=1,\,f_{\partial\Lambda}=0\bigr\}.
\end{equation}
Obviously, $\gg_{\Lambda}^{(i)}$ is continuous in $\omega_b$ for $b:=\langle x,x+\hate_i\rangle$ and it is bounded away from~$0$ and infinity. (In Section~\ref{sec-perturbed-coordinate} we will see that $\gg_{\Lambda}^{(i)}$ is in fact a double gradient of the Green function for operator~$\LL_\omega$ on~$\Lambda$.) 
Note that \eqref{E:2.10} and \eqref{E:2.4} ask us to understand how $\nabla_i\Psi_\Lambda(\omega,x)$ changes when the coordinate of $\omega$ over $\langle x,x+\hate_i\rangle$ is perturbed. This change takes a multiplicative form:

\begin{proposition}[Rank-one perturbation]
\label{prop-perturbed-corrector}
Let $\Lambda\subset\Z^d$ be finite and $x,y\in\Lambda$ be nearest neighbors; $y=x+\hate_i$ for some $i\in\{1,\dots,d\}$. For any $\omega,\omega'$ that agree everywhere except at edge $b:=\langle x,y\rangle$,
\begin{equation}
\label{E:2.21}
\nabla_i\Psi_\Lambda(\omega',x)=
\bigl[1-(\omega_b'-\omega_b)\gg_{\Lambda}^{(i)}(\omega',x)\bigr]\nabla_i\Psi_\Lambda(\omega,x).
\end{equation}
For the prefactor we get the alternative expressions
\begin{equation}
\label{E:2.24b}
\begin{aligned}
1-(\omega_b'-\omega_b)\gg_{\Lambda}^{(i)}(\omega',x)
=\bigl(1+(\omega_b'-\omega_b)\gg_{\Lambda}^{(i)}(\omega,x)\bigr)^{-1}
=\frac{\gg_{\Lambda}^{(i)}(\omega',x)}{\gg_{\Lambda}^{(i)}(\omega,x)}.
\end{aligned}
\end{equation}
In particular, the factor $1-(\omega_b'-\omega_b)\gg_{\Lambda}^{(i)}(\omega',x)$ is bounded away from 0 and $\infty$ uniformly in $\omega\in\Omega$ and $\Lambda\subset\Z^d$.
\end{proposition}

It is worthy a note that \eqref{E:2.21} is a direct consequence of a rank-one perturbation formula for the Green function; cf Lemma~\ref{lemma-5.1}. Incidentally, such formulas have proved extremely useful in the analysis of random Schr\"odinger operators, including those associated with operator~$\cmss L_\omega$ (Aizenman and Molchanov~\cite{Aizenman-Molchanov}). What matters for us is the $\Lambda\uparrow\Z^d$-limit of the expression in \eqref{E:2.24b} can be controlled uniformly in $\omega\in\Omega$:

\begin{proposition}
\label{prop-2.8}
Recall that $\Omega:=[\lambda,1/\lambda]^{\B(\Z^d)}$ for some~$\lambda\in(0,1)$. Then $\Lambda\mapsto\gg^{(i)}_\Lambda(\omega,x)$ is non-decreasing and bounded away from zero and infinity uniformly in $\Lambda\subset\Z^d$ and~$\omega\in\Omega$. Moreover, for all $\omega\in\Omega$ and all $x\in\Z^d$ the limit
\begin{equation}
\label{E:2.22}
\gg^{(i)}(\omega,x):=\lim_{\Lambda\uparrow\Z^d}\gg^{(i)}_\Lambda(\omega,x)
\end{equation}
exists and satisfies
\begin{equation}
\label{E:2.25a}
\gg^{(i)}(\omega,x)^{-1}=\inf\bigl\{Q_{\Z^d}(f)\colon 0\le f\le 1,\,f(x+\hate_i)-f(x)=1,\,|\supp(f)|<\infty\bigr\},
\end{equation}
where $\supp(f):=\{x\in\Z^d\colon f(x)\ne0\}$. In particular, $\omega\mapsto\gg^{(i)}(\omega,x)$ is monotone and continuous in the product topology on~$\Omega$. 
Finally, $\omega_b\mapsto\gg^{(i)}(\omega,x)$ for $b:=\langle x,x+\hate_i\rangle$ is continuous and bounded and $(\omega,x)\mapsto \gg^{(i)}(\omega,x)$ is stationary in the sense that $\gg^{(i)}(\tau_z\omega,x+z)=\gg^{(i)}(\omega,x)$ holds for all $\omega\in\Omega$ and all $x,z\in\Z^d$.
\end{proposition}

Before we wrap up the outline of the proof of Theorem~\ref{thm1}, let us formulate a representation for the limiting variance $\sigma_t^2$ from Theorem~\ref{thm1}: For $x\in\Z^d$ and $i\in\{1,\dots,d\}$, let $b$ denote the edge corresponding to the pair $(x,i)$ and let
\begin{equation}
\label{eqn-hdef}
h(\omega,x,i):=\int\BbbP(\textd\omega_b')\int_{\omega_b'}^{\omega_b}
\textd\tilde\omega_b\,\,\bigl[1-(\tilde\omega_b-\omega_b)\gg^{(i)}(\tilde\omega,x)\bigr]^2,
\end{equation}
where $\tilde\omega$ is the configuration equal to $\omega$ except at $b$, where it equals $\tilde\omega_b$. 
 Define the matrix $\hat Z(x,i):=\{\hat Z_{jk}(x,i)\}_{j,k=1,\dots,d}$ by the quadratic form
\begin{equation}
\label{E:2.27}
\bigl(t,\hat Z(x,i) t\bigr):=\E\Bigl(h(\cdot,x,i)\bigl|\nabla_i(t\cdot\psi)(\cdot,x)\bigr|^2\,\Big|\,\sigma\bigl(\omega_{b'}\colon b'\preceq (x,i)\bigr)\Bigr),
\end{equation}
where $(x,i)$ represents the edge $\langle x,x+\hate_i\rangle$ and $t\in\R^d$. Then we have:

\begin{theorem}[Limiting variance]
\label{thm2}
Under the assumptions of Theorem~\ref{thm1}, the matrix elements of $\hat Z(x,i)$ are square integrable. In particular, $\sigma_t^2$ from Theorem~\ref{thm1} is given by
\begin{equation}
\label{eqn-sigmadef}
\sigma_t^2 = \sum_{i=1}^d \,\E\Bigl(\bigl(t,\hat Z(0,i) t\bigr)^2\Bigr),\qquad t\in\R^d.
\end{equation}
\end{theorem}

As an inspection of \eqref{E:2.27} reveals, the limiting variance is thus a bi-quadratic form in $t$. Although concisely written, the expression is not very useful from the practical point of view; particularly, due to the unwieldy conditioning in~\eqref{E:2.27}. Notwithstanding, a slightly more explicit form can be obtained by performing the inner integral in \eqref{eqn-hdef} with the help of \eqref{E:2.24b}:
\begin{equation}
h(\omega,x,i)=\int\BbbP(\textd\omega_b')(\omega_b-\omega_b')\frac{\gg^{(i)}(\omega',x)}{\gg^{(i)}(\omega,x)},
\end{equation}
where, as before, $\omega'$ agrees with~$\omega$ except at~$b:=\langle x,x+\hate_i\rangle$. Note that this implies, at least formally, that
\begin{equation}
\label{E:2.31ww}
h(\omega,x,i)\bigl|\nabla\psi(\omega,x)\bigr|^2 = \int\BbbP(\textd\omega_b')(\omega_b-\omega_b')\bigl|\nabla\psi(\omega,x)\bigr|\,\bigl|\nabla\psi(\omega',x)\bigr|,
\end{equation}
and so
\begin{equation}
\E\Bigl((t,\hat Z(x,i)t)\,\Big|\,\sigma\bigl(\omega_{b'}\colon b'\prec (x,i)\bigr)\Bigr)
=0,
\end{equation}
i.e., that $(t,\hat Z(x,i)t)$ is a martingale increment. A question of interest is whether an expression can be found for $\sigma_t^2$ that is more amenable to computations.

\begin{remark}
Since $t\mapsto C_L^{\eff}(t)$ is quadratic in~$t$, and thus linear in $\{t_it_j\colon 1\le i\le j\le d\}$, Theorem~\ref{thm2} and (a version of) the Cram\'er-Wold device imply that, as $L\to\infty$, the joint law of
\begin{equation}
\biggl\{\frac{C_L^{\,\eff}(t)-\E C_L^{\,\eff}(t)}{|\Lambda_L|^{1/2}}\colon t\in\R^d\biggr\}
\end{equation}
tends to a multivariate Gaussian $\{G_t\colon t\in\R^d\}$ with
\begin{equation}
E(G_t)=0\quad\text{and}\quad E\bigl(G_tG_s)=\sum_{i=1}^d \,\E\Bigl(\bigl(t,\hat Z(0,i) t\bigr)\bigl(s,\hat Z(0,i) s\bigr)\Bigr),
\end{equation}
where $\hat Z(0,i)$ is as in \eqref{E:2.27}. Naturally, $t\mapsto G_t$ is a quadratic form as well.
\end{remark}

\subsection{Organization}
The proofs (and the rest of the paper) are organized as follows. In Section~\ref{sec3} we assemble the ingredients --- following the steps outlined in the present section --- into the proofs of Theorems~\ref{thm1} and~\ref{thm2} and Corollary~\ref{cor1.2}. In Section~\ref{sec-Meyers} we then show that the finite-volume harmonic coordinate approximates its full lattice counterpart in an $L^2$-sense as stated in Proposition~\ref{prop-approx-corrector-L2} and establish the Meyers estimate from Proposition~\ref{prop:Meyers estimate}. A key technical tool is the Calder\'on-Zygmund regularity theory and a uniform bound on the triple gradient of the Green function of the simple random walk in finite boxes. Finally, in Section~\ref{sec-perturbed-coordinate}, we prove Propositions~\ref{prop-perturbed-corrector} and \ref{prop-2.8} dealing with the harmonic coordinate over environments perturbed at a single edge.

\section{Proof of the CLT}
\label{sec3}\noindent
In this section we verify the conditions \twoeqref{E:2.4a}{E:2.5a} of the Martingale Central Limit Theorem and thus prove Theorems~\ref{thm1} and~\ref{thm2}. All derivations are conditional on Propositions~\ref{prop-approx-corrector-L2}--\ref{prop-2.8} the proofs of which are postponed to later sections. Throughout we will make use of the following simple but useful consequence of H\"older's inequality:

\begin{lemma}
\label{lemma:p_p'}
For any $p'>p>2$, $\alpha:=\frac2p\frac{p'-p}{p'-2}$ and $\beta:=\frac {p'}p\frac{p-2}{p'-2}$,
\begin{equation}
\bigr\Vert\nabla(\Psi_{\Lambda_L}-\psi)\bigr\Vert_{\Lambda_L,p}
\leq\bigr\Vert\nabla(\Psi_{\Lambda_L}-\psi)\bigr\Vert_{\Lambda_L,2}^\alpha\,
\bigr\Vert\nabla(\Psi_{\Lambda_L}-\psi)\bigr\Vert_{\Lambda_L,p'}^\beta.
\end{equation}
\end{lemma}
\begin{proofsect}{Proof}
Apply H\"older's inequality to the function $f:=|\nabla(\Psi_{\Lambda_L}-\psi)|$.
\end{proofsect}

Assume now the setting developed in Section~\ref{sec2}; in particular, the ordering of edges and sigma-algebras $\scrF_k$ from Section~\ref{sec2.2} and the martingale increment~$Z_k$ from~\eqref{E:2.3} and its representation \eqref{E:2.10} from Section~\ref{sec2.3}. In analogy with equation \eqref{eqn-hdef}, we also define 
\begin{equation}
\label{eqn-hlambdadef}
h_\Lambda(\omega,x,i):=\int\BbbP(\textd\omega_b')\int_{\omega_b'}^{\omega_b}
\textd\tilde\omega_b\bigl[1-(\tilde\omega_b-\omega_b)\gg_\Lambda^{(i)}(\tilde\omega,x)\bigr]^2,
\end{equation}
where $b:=\langle x,x+\hate_i\rangle$ and $\tilde\omega$ is the configuration equal to $\omega$ except at $b$, where it equals $\tilde\omega_b$. By Proposition~\ref{prop-perturbed-corrector}, we may write the martingale increment $Z_k$ as
\begin{equation}
\label{E:3.2a}
Z_k=\E\Bigl(h_{\Lambda_L}(\cdot,x_k,i_k)\bigl|\nabla_{i_k}(t\cdot\Psi_{\Lambda_L})(\cdot,x_k)\bigr|^2\,\Big|\FF_{k}\Bigr),
\end{equation}
where $x_k$ and $i_k$ are the vertex and the edge direction corresponding to $b_k$, i.e., $b_k=\langle x_k,x_k+\hate_{i_k}\rangle$. (A representation similar to \eqref{E:2.31ww} is possible here as well.) Recall the notation for $\hat Z(x,i)$ from \eqref{E:2.27} and note that this is well defined and finite $\BbbP$-a.s.\ thanks to the estimates \twoeqref{eqn-approx-corrector-Lp-1}{eqn-approx-corrector-Lp-2} as well as boundedness of~$h$. Note the dependence of $Z_k$ on $L$.


\begin{proposition}[Martingale CLT --- first condition]
\label{prop-MCLT1}
Assume that the premises (and thus conclusions) of Propositions~\ref{prop-approx-corrector-L2}--\ref{prop-2.8} hold. Then $Z_k\in L^2(\BbbP)$ for all~$k$ and
\begin{equation}
\label{E:3.3}
\frac1{|\Lambda_L|}\sum_{k=1}^{|\B(\Lambda_L)|}
\E(Z_k^2|\FF_{k-1})\,\,\underset{L\to\infty}\longrightarrow\,\,\sum_{i=1}^d \E\Bigl(\bigl(t,\hat Z(0,i) t\bigr)^2\Bigr)
\end{equation}
in~$\BbbP$-probability and $L^1(\BbbP)$.
\end{proposition}

\begin{proofsect}{Proof}
Fix~$t\in\R^d$. Thanks to Lemma~\ref{lemma-2.3} and Proposition~\ref{prop-corrector}(2), for each~$i\in\{1,\dots,d\}$, the collection of conditional expectations
\begin{equation}
\biggl\{\E\Bigl(\,\bigl(\,t,\hat Z(x,i)t\,\bigr)^2\,\Big|\,\sigma\bigl(\omega_b\colon b\prec(x,i)\bigr)\,\Bigr)\colon x\in\Z^d\biggr\}
\end{equation}
is stationary with respect to the shifts on~$\Z^d$ and, by Proposition~\ref{prop:Meyers estimate}, bounded in~$L^1(\BbbP)$. Labeling the edges in~$\B(\Lambda_L)$ according to the order $\preceq$, the Spatial Ergodic Theorem yields
\begin{equation}
\frac1{|\Lambda_L|}\sum_{k=1}^{|\B(\Lambda_L)|}
\E\bigl((t,\hat Z(x_k,i_k)t)^2\big|\FF_{k-1}\bigr)\,\,\underset{L\to\infty}\longrightarrow\,\,\sum_{i=1}^d \E\Bigl(\bigl(t,\hat Z(0,i) t\bigr)^2\Bigr)
\end{equation}
with the limit $\BbbP$-a.s.\ and in $L^1(\BbbP)$. To see how this relates to our claim, abbreviate
\begin{align}
A_k &:= h_{\Lambda_L}(\cdot,x_k,i_k)\bigl|\nabla_{i_k}(t\cdot\Psi_{\Lambda_L})(\cdot,x_k)\bigr|^2,
\label{E:3.6}
\\
B_k &:= h(\cdot,x_k,i_k)\bigl|\nabla_{i_k}(t\cdot\psi)(\cdot,x_k)\bigr|^2,
\label{E:3.7}
\end{align}
and denote
\begin{align}
R_{L,k}:=\E\Big[\E\big[A_k\,\big|\FF_{k}\big]^2- \E\big[B_k\,\big|\FF_{k}\big]^2 \Big|\FF_{k-1}\Big].
\end{align}
By \eqref{E:3.2a} we have $Z_k=\E(A_k\,\big|\FF_{k})$, while \eqref{E:2.27} reads $(t,\hat Z(x_k,i_k)t)=\E(B_k\,\big|\FF_{k})$. Hence, as soon as we show that
\begin{equation}
\label{E:3.9}
\frac1{|\Lambda_L|}\sum_{k=1}^{|\B(\Lambda_L)|}
\E\bigl(|R_{L,k}|\bigr)\,\,\underset{L\to\infty}\longrightarrow\,\,0,
\end{equation}
the claim \eqref{E:3.3} will follow.

The proof of \eqref{E:3.9} will proceed by estimating $\E|R_{L,k}|$ which will involve applications of the Cauchy-Schwarz inequality (in order to separate terms) and Jensen's inequality (in order to eliminate conditional expectations). First we note
\begin{equation}
\E|R_{L,k}|
\le\bigl(\E\bigl[(A_k-B_k)^2\bigr]\bigr)^{\ffrac12}\bigl(\E\bigl[(A_k+B_k)^2\bigr]\bigr)^{\ffrac12}.
\end{equation}
Writing $A_k=B_k+(A_k-B_k)$ and noting $(a+b)^2\le 2a^2+2b^2$ tells us
\begin{equation}
\E\bigl[(A_k+B_k)^2\bigr]
\le2\E\bigl[(A_k-B_k)^2\bigr]+8\E\bigl(B_k^2\bigr).
\end{equation}
Summing over $k$ and applying the Cauchy-Schwarz inequality, we find that
\begin{equation}
\label{E:3.13}
\frac1{|\Lambda_L|}\sum_{k=1}^{|\B(\Lambda_L)|}
\E\bigl(|R_{L,k}|\bigr)
\le\sqrt{\alpha\bigl(2\alpha+8\beta\bigr)},
\end{equation}
where
\begin{equation}
\alpha:=\frac1{|\Lambda_L|}\sum_{k=1}^{|\B(\Lambda_L)|}\E\bigl[(A_k-B_k)^2\bigr]\quad\text{and}\quad
\beta:=\frac1{|\Lambda_L|}\sum_{k=1}^{|\B(\Lambda_L)|}\E(B_k^2).
\end{equation}
By inspection of \eqref{E:3.13} we now observe that it suffices to show that $\beta$ stays bounded while $\alpha$ tends to zero in the limit $L\to\infty$.

The boundedness of $\beta$ follows from \eqref{eqn-approx-corrector-Lp-1} and the fact that $h(\cdot,x,i)$ is bounded; indeed, these yield $\E(|B_k|^2)\le\Vert h\Vert_\infty^2|t|^4\Vert\nabla\psi\Vert_4^4$ uniformly in~$k$ and~$L$. Concerning the terms constituting $\alpha$, using $(a+b)^2\le2a^2+2b^2$ we first separate terms as
\begin{multline}
\label{E:3.14d}
\qquad
\E\bigl[(A_k-B_k)^2\bigr]\le2\E\Bigl(\bigl|h_{\Lambda_L}(\cdot,x_k,i_k)\bigr|^2\bigl|\,|\nabla_{i_k}(t\cdot\Psi_{\Lambda_L})(\cdot,x_k)|^2-|\nabla_{i_k}(t\cdot\psi)(\cdot,x_k))|^2\bigr|^2\Bigr)
\\
+\,
2\E\Bigl(\bigl|h_{\Lambda_L}(\cdot,x_k,i_k)-h(\cdot,x_k,i_k)\bigr|^2\bigl|\nabla_{i_k}(t\cdot\psi)(\cdot,x_k)\bigr|^4\Bigr).
\qquad
\end{multline}
Since $h_\Lambda$ is uniformly bounded, the average over $k$ of the first term is bounded by a constant times the product of $(\Vert\nabla\Psi_{\Lambda_L}\Vert_{\Lambda_L,4}+\Vert\nabla\psi\Vert_{\Lambda_L,4})^2$ and $\Vert\nabla(\Psi_{\Lambda_L}-\psi)\Vert_{\Lambda_L,4}^2$. The latter tends to zero as $L\to\infty$ by Proposition~\ref{prop:Meyers estimate}, Proposition~\ref{prop-approx-corrector-L2} and Lemma~\ref{lemma:p_p'} (with the choices $p:=4$ and $p'>4$ but sufficiently close to $4$). 

For the second term in \eqref{E:3.14d} we pick $p>4$ and use H\"older's inequality to get
\begin{multline}
\label{E:3.16}
\qquad
\frac1{|\Lambda_L|}\sum_{k=1}^{|\B(\Lambda_L)|}
\E\Bigl(\bigl|h_{\Lambda_L}(\cdot,x_k,i_k)-h(\cdot,x_k,i_k)\bigr|^2\bigl|\nabla_{i_k}(t\cdot\psi)(\cdot,x_k)\bigr|^4\Bigr)
\\
\le
|t|^4\,\Vert\nabla\psi\Vert_{\Lambda_L,p}^4\biggl(\,\frac1{|\Lambda_L|}\sum_{k=1}^{|\B(\Lambda_L)|}
\E\Bigl(\bigl|h_{\Lambda_L}(\cdot,x_k,i_k)-h(\cdot,x_k,i_k)\bigr|^{2q}\Bigr)\biggr)^{\ffrac1q},
\qquad
\end{multline}
where $q$ satisfies $\ffrac4p+\ffrac1q=1$. The norm of $\Vert\nabla\psi\Vert_{\Lambda_L,p}$ is again bounded by Proposition~\ref{prop:Meyers estimate} as long as~$p$ is sufficiently close to~$4$; to apply \eqref{eqn-approx-corrector-Lp-1}, we need to invoke the stationarity of $\nabla\psi$ to realize $\Vert\nabla\psi\Vert_{\Lambda_L,p}=\Vert\nabla\psi\Vert_p$.

For the second term in \eqref{E:3.16} we first need to show that for each~$\epsilon>0$ there is $N\ge1$ so that for all $\omega\in\Omega$,
\begin{equation}
\label{E:3.16c}
\dist_{\ell^1(\Z^d)}(x,\Lambda_L^\cc)\ge N\quad\Rightarrow\quad \bigl|\,h_{\Lambda_L}(\omega,x,i)-h(\omega,x,i)\bigr|<\epsilon.
\end{equation}
For this we use that, thanks to \eqref{eqn-hdef}, \eqref{eqn-hlambdadef}, \eqref{E:1.3} and the monotonicity of $\Lambda\mapsto\gg_\Lambda^{(i)}(\tilde\omega,x)$,
\begin{equation}
\label{E:3.17c}
\bigl|\,h_\Lambda(\omega,x,i)-h(\omega,x,i)\bigr|
\le C\int_\lambda^{\ffrac1\lambda}\textd\tilde\omega_b\,\bigl(\gg^{(i)}(\tilde\omega,x)-\gg_\Lambda^{(i)}(\tilde\omega,x)\bigr)
\end{equation}
for some constant $C=C(\lambda)<\infty$. To estimate the right-hand side, we invoke stationarity with respect to shifts and note that whenever $x+\Lambda_N\subset\Lambda$, we have
\begin{equation}
\label{E:3.18c}
\gg^{(i)}(\omega,x)-\gg_\Lambda^{(i)}(\omega,x)
\le \gg^{(i)}(\tau_x\omega,0)-\gg_{\Lambda_N}^{(i)}(\tau_x\omega,0),\qquad \omega\in\Omega.
\end{equation}
Then \eqref{E:3.16c} follows from \eqref{E:3.17c} and the fact that the difference on the right-hand side of \eqref{E:3.18c} converges to zero uniformly in $\omega\in\Omega$. (Specifically, we apply Dini's theorem for uniformity: $\Omega$ is compact in the product topology by Tychonoff's theorem, $L\mapsto\gg_{\Lambda_N}^{(i)}(\cdot,0)$ is a non-decreasing sequence of continuous functions and the limit $\gg^{(i)}(\cdot,0)$ is continuous as well.)

We now bound the last term in \eqref{E:3.16} as follows. The terms for which $x_k$ is at least~$N$ steps away from $\Lambda_L$ are bounded by~$\epsilon$ thanks to \eqref{E:3.17c}; the sum over the remaining terms is of order~$NL^{d-1}$ thanks to the uniform boundedness of $h_\Lambda-h$. Hence, in the limit $L\to\infty$, the second term in \eqref{E:3.16} is of order~$\epsilon^{\ffrac1q}$; taking $\epsilon\downarrow0$ shows that~$\alpha$ tends to zero as $L\to\infty$. Invoking \eqref{E:3.13}, this finishes the proof of \eqref{E:3.9} and the whole claim.
\end{proofsect}

\begin{proposition}[Martingale CLT --- second condition] 
\label{prop-MCLT2}
Assume that the premises (and thus conclusions) of Propositions~\ref{prop-approx-corrector-L2}--\ref{prop-2.8} hold. Then for each $\epsilon>0$,
\begin{equation}
\frac1{|\Lambda_L|}\sum_{k=1}^{|\B(\Lambda_L)|}
\E\Bigl(\,Z_k^2\1_{\{|Z_k|>\epsilon|\Lambda_L|^{1/2}\}}\Big|\FF_{k-1}\Bigr)\,\underset{L\to\infty}\longrightarrow\,0,
\end{equation}
in $\BbbP$-probability.
\end{proposition}

\begin{proofsect}{Proof}
This could be proved by strengthening a bit the statement of Proposition~\ref{prop-MCLT1} (from squares of the $Z$'s to a slightly higher power), but a direct argument is actually easier.

First we note that it suffices to show convergence in expectation. Let $p>4$ be such that the statements in Proposition~\ref{prop:Meyers estimate} hold. By H\"older's and Chebyshev's inequalities we have
\begin{equation}
\label{E:3.19}
\E\Bigl(\,Z_k^2\1_{\{|Z_k|>\epsilon|\Lambda_L|^{1/2}\}}\Bigr)\le\Bigl(\frac1{\epsilon|\Lambda_L|^{1/2}}\Bigr)^{\frac{p-4}2}\E\bigl(\,|Z_k|^{p/2}\bigr).
\end{equation}
Since $h_{\Lambda_L}$ is bounded, Jensen's inequality yields 
\begin{equation}
\E\bigl(\,|Z_k|^{p/2}\bigr)
\le C\E\biggl(\Bigl[\E\Bigl(\bigl|\nabla_{i_k}(t\cdot\Psi_\Lambda)(\cdot,x_k)\bigr|^2\,\Big|\FF_{k}\Bigr)\Bigr]^{\ffrac p2}\biggr)\le C\E\Bigl(\bigl|\nabla_{i_k}(t\cdot\Psi_\Lambda)(\cdot,x_k)\bigr|^p\Bigr).
\end{equation}
It follows that
\begin{equation}
\label{E:3.22q}
\frac1{|\Lambda_L|}\sum_{k=1}^{|\B(\Lambda_L)|}\E\bigl(\,|Z_k|^{p/2}\bigr)\le C|t|^{p}\Vert\nabla\Psi_{\Lambda_L}\Vert_{\Lambda_L,p}^p.
\end{equation}
The right-hand side is bounded uniformly in~$L$. Using this in \eqref{E:3.19}, the claim follows.
\end{proofsect}

We can now finish the proof of our main results:

\begin{proofsect}{Proof of Theorems~\ref{thm1} and~\ref{thm2} from Propositions~\ref{prop-approx-corrector-L2}--\ref{prop-2.8}}
The distributional convergence in \eqref{E:2.1} is a direct consequence of the Martingale Central Limit Theorem whose conditions \twoeqref{E:2.4a}{E:2.5a} are established in Propositions~\ref{prop-MCLT1} and~\ref{prop-MCLT2}. The limiting variance $\sigma_t^2$ is given by the right-hand side of \eqref{E:3.3}, in agreement with \eqref{eqn-sigmadef}. It remains to prove that $\sigma_t^2>0$ whenever $t\ne0$ and the law~$\BbbP$ is non-degenerate.

Suppose on the contrary that $\sigma_t^2=0$. Then for each~$i$ we would have $\E((t,\hat Z(0,i) t)^2)=0$ and thus $(t,\hat Z(0,i)t)=0$ $\BbbP$-a.s. Denoting $b:=\langle0,\hate_i\rangle$, \twoeqref{eqn-hdef}{E:2.27} imply that, for $\BbbP$-a.e.~$\omega_b$,
\begin{equation}
\label{E:3.23d}
\int\BbbP(\textd\omega_b')\int_{\omega_b'}^{\omega_b}\textd\tilde\omega_b\,\,\E\Bigl(\bigl[1-(\tilde\omega_b-\omega_b)\gg_\Lambda^{(i)}(\tilde\omega,0)\bigr]\bigl|\nabla_i(t\cdot\psi)(\omega,0)\bigr|^2\Big|\FF_{(0,i)}\Bigr)=0,
\end{equation}
where $\FF_{(0,i)}:=\sigma(\omega_{b})$. Let $\Omega_1\subset[\lambda,\ffrac1\lambda]$ be the set of $\omega_b$ where this holds. Then~$\BbbP(\Omega_1)=1$ and, since~$\BbbP$ is non-degenerate, $\Omega_1$ contains at least two points. The expectation in \eqref{E:3.23d} is independent of~$\omega_b'$; subtracting the expression for two (generic) choices of~$\omega_b$ in~$\Omega_1$ then shows that the inner integral must vanish for all $\omega_b,\omega_b'\in\Omega_1$. But \eqref{E:2.24b} tells us that the prefactor in square brackets, and thus the conditional expectation, is non-negative (in fact, it is bounded away from zero). Hence, this can only happen when
\begin{equation}
\nabla_i(t\cdot\psi)(\cdot,0)=0,\qquad\BbbP\text{-a.s.\ for all }i=1,\dots,d.
\end{equation}
But then $c_{\eff}(t)=0$, which cannot hold for $t\ne0$ when \eqref{E:1.3} is in force. 
\end{proofsect}

\begin{proofsect}{Proof of Corollary~\ref{cor1.2} from Propositions~\ref{prop-approx-corrector-L2}--\ref{prop-2.8}}
Thanks to \twoeqref{E:2.1q}{E:2.3} and Proposition~\ref{prop-MCLT1}, $C_L^{\,\eff}(t)$ is a martingale whose increments, $Z_k$ are square integrable. Therefore,
\begin{equation}
\Var\bigl(C_L^{\,\eff}(t)\bigr)=\sum_{k=1}^{|\B(\Lambda_L)|}\E(Z_k^2).
\end{equation}
But the right-hand side is the expectation of the quantity on the left of \eqref{E:3.3}. Since the convergence in \eqref{E:3.3} occurs in $L^1(\BbbP)$, the claim follows.
\end{proofsect}

\section{The Meyers estimate}
\label{sec-Meyers}\noindent
The goal of this section is to give proofs of Propositions~\ref{prop-approx-corrector-L2} and~\ref{prop:Meyers estimate}. The former is a simple consequence of the Hilbert-space structure underlying the definition of a harmonic coordinate; the latter (to which this section owes its name) is a consequence of the Calder\'on-Zygmund regularity theory for singular integral operators.

\subsection{$L^2$ bounds and convergence}
Recall our notation $\LL_\omega$ for the operator in \eqref{E:2.8a}. We begin by noting an explicit representation of the minimum of $f\mapsto Q_\Lambda(f)$ as a function of the (Dirichlet) boundary condition:

\begin{lemma}
\label{lemma-4.1}
Let $\Lambda\subset\Z^d$ be finite and fix an $\omega\in\Omega$. Then there is $K\colon\partial\Lambda\times\partial\Lambda\to[0,\infty)$, depending on $\Lambda$ and $\omega$, such that for any $h$ that obeys $\LL_\omega h(x)=0$ for $x\in\Lambda$, 
\begin{equation}
Q_\Lambda(h)=\frac12\sum_{x,y\in\partial\Lambda}K(x,y)\bigl[h(y)-h(x)\bigr]^2.
\end{equation}
Moreover, $K(x,y)=K(y,x)$ for all $x,y\in\partial\Lambda$ and
\begin{equation}
\sum_{y\in\partial\Lambda}K(x,y)=\sum_{\begin{subarray}{c}
z\in\Lambda\\
\langle x,z\rangle\in\B(\Lambda)
\end{subarray}}
a_{xz}
\end{equation}
for all $x\in\partial\Lambda$.
\end{lemma}

\begin{proofsect}{Proof}
``Integrating'' by parts we obtain
\begin{equation}
\begin{aligned}
\label{eqn:bound_L2}
Q_\Lambda(h)
&=-\sum_{y\in\Lambda}h(y)(\LL_\omega h)(y)+
\sum_{\begin{subarray}{c}
y\in\partial\Lambda,\,x\in\Lambda\\
\langle x,y\rangle\in\B(\Lambda)
\end{subarray}}
 a_{xy}\,\bigl[h(y)-h(x)\bigr]h(y).
\end{aligned}
\end{equation}
Employing the fact that $h$ is $\LL_\omega$-harmonic, the first sum drops out. For the second sum we recall that $h(x)=\sum_{z\in\partial \Lambda}p_\Lambda(x,z)h(z)$, where $p_\Lambda(x,z)$ is the discrete Poisson kernel which can be defined e.g.\ by $p_\Lambda(x,z):=P_\omega^x(X_{\tau_{\partial\Lambda}}=z)$ for $\tau_{\partial\Lambda}$ denoting the first exit time from $\Lambda$ of the random walk in conductances~$\omega$. Now set
\begin{equation}
K(y,z):=\sum_{\begin{subarray}{c}
x\in\Lambda\\\langle x,y\rangle\in\B(\Lambda)
\end{subarray}}
a_{xy}p_\Lambda(x,z)
\end{equation}
and note that $\sum_{z\in\partial\Lambda}K(y,z)=\sum_{x\in\Lambda,\,\langle x,y\rangle\in\B(\Lambda)}a_{xy}$. It follows that
\begin{equation}
\sum_{\begin{subarray}{c}
y\in\partial\Lambda,\,x\in\Lambda\\
\langle x,y\rangle\in\B(\Lambda)
\end{subarray}}
 a_{xy}\,\bigl[h(y)-h(x)\bigr]h(y)
=\sum_{y,z\in\partial\Lambda} K(y,z)\bigl[h(y)-h(z)\bigr]h(y).
\end{equation}
The representation using the random walk and its reversiblity now imply that $K$ is symmetric. Symmetrizing the last sum then yields the result. 
\end{proofsect}

\begin{remark}
We note that Lemma~\ref{lemma-4.1} holds even for vector valued functions; just replace $[h(y)-h(x)]^2$ by the norm squared of $h(y)-h(x)$. This applies to several derivations that are to follow; a point that we will leave without further comment.
\end{remark}

We can now prove Proposition \ref{prop-approx-corrector-L2} dealing with the convergence of $\nabla\Psi_\Lambda$ to~$\nabla\psi$ in $\Vert\cdot\Vert_{\Lambda,2}$-norm, as $\Lambda:=\Lambda_L$ fills up all of~$\Z^d$.
 
\begin{proofsect}{Proof of Proposition~\ref{prop-approx-corrector-L2}}
Abbreviate $h(x):=\psi(\omega,x)-\Psi_{\Lambda_L}(\omega,x)$. The bound \eqref{E:1.3} implies
\begin{equation}
\label{E:4.5}
\bigr\Vert\nabla(\Psi_{\Lambda_L}-\psi)\bigr\Vert_{\Lambda_L,2}^2\le\frac1\lambda\,\frac1{|\Lambda_L|}\,\,
\E\biggl(\sum_{\langle x,y\rangle\in\B(\Lambda_L)}a_{xy}\bigl|h(y)-h(x)\bigr|^2\biggr).
\end{equation}
Let $f\colon\Lambda\cup\partial\Lambda\to\R^d$ be the minimizer of
\begin{equation}
\inf\biggl\{\sum_{\langle x,y\rangle \in\mathbb B(\Lambda_L)} \bigl|f(y)-f(x)\bigr|^2,\,f(z)=\chi(z)\mbox{ for all } z\in\partial\Lambda_L\biggr\}.
\end{equation}
Since $h$ is the minimizer of the corresponding Dirichlet energy with conductances $\{a_{xy}\}$ and boundary condition~$\chi$, we get using \eqref{E:1.3}
\begin{equation}
\begin{aligned}
\sum_{\langle x,y\rangle \in\mathbb B(\Lambda_L)} a_{xy}\bigl|h(y)-h(x)\bigr|^2&\leq
\sum_{\langle x,y\rangle \in\mathbb B(\Lambda_L)} a_{xy}\bigl|f(y)-f(x)\bigr|^2
\\
&\leq \frac 1\lambda\sum_{\langle x,y\rangle \in\mathbb B(\Lambda_L)} \bigl|f(y)-f(x)\bigr|^2.
\end{aligned}
\end{equation}
Writing the last sum coordinate-wise and applying Lemma~\ref{lemma-4.1}, we thus get
\begin{equation}
\label{E:4.8}
\sum_{\langle x,y\rangle \in\mathbb B(\Lambda_L)} a_{xy}\bigl|h(y)-h(x)\bigr|^2
\le\frac1{2\lambda}\sum_{x,y\in\partial\Lambda_L}K(x,y)\bigl|\chi(\omega,y)-\chi(\omega,x)\bigr|^2,
\end{equation}
where the kernel $K(x,y)$ pertains to the homogeneous problem, i.e., the simple random walk. Note that these bounds hold for all configurations satisfying \eqref{E:1.3}.

By shift covariance and sublinearity of the corrector (cf~Proposition~\ref{prop-corrector}(2,4)), for each $\varepsilon>0$ there is $A=A(\varepsilon)$ such that
\begin{equation}
\mathbb E\bigl(\,\bigl|\chi(\cdot,x)-\chi(\cdot,y)\bigr|^2\bigr)\leq A+\varepsilon |x-y|^2.
\end{equation}
Using this and \eqref{E:4.8} in \eqref{E:4.5} yields
\begin{equation}
\label{E:4.10}
\bigr\Vert\nabla(\Psi_{\Lambda_L}-\psi)\bigr\Vert_{\Lambda_L,2}^2\le\frac1{2\lambda^2}\,\frac1{|\Lambda_L|}
\sum_{x,y\in\partial\Lambda_L}K(x,y)\bigl(A+\varepsilon|x-y|^2\bigr).
\end{equation}
But $\sum_{y\in\partial\Lambda_L}K(x,y)\le1$ for each $x\in\partial\Lambda_L$ while $\sum_{x,y\in\partial\Lambda_L}K(x,y)|x-y|^2$ is, by Lemma~\ref{lemma-4.1}, the Dirichlet energy of the function $x\mapsto x$ for conductances all equal to~$1$. Hence, the last sum in \eqref{E:4.10} is bounded by $A|\partial\Lambda_L|+\varepsilon|\B(\Lambda_L)|$. Taking $L\to\infty$ and $\varepsilon\downarrow0$ finishes the proof.
\end{proofsect}

\begin{remark}
\label{R:representation}
As alluded to in the introduction, the $L^2$-convergence $\nabla\Psi_{\Lambda_L}\to\nabla\psi$ permits us to prove the formula \eqref{E:1.5} for~$c_\eff(t)$. The argument is similar to (albeit much easier than) what we used in the proof of Proposition~\ref{prop-MCLT1}. Indeed, we trivially decompose
\begin{equation}
\label{E:4.11}
C^{\,\eff}_L(t)=Q_{\Lambda_L}\bigl(t\cdot\Psi_{\Lambda_L}\bigr)=Q_{\Lambda_L}(t\cdot\psi)+\bigl(Q_{\Lambda_L}\bigl(t\cdot\Psi_{\Lambda_L}\bigr)-Q_{\Lambda_L}(t\cdot\psi)\bigr).
\end{equation}
The stationarity of the gradients of~$\psi$ and the Spatial Ergodic Theorem imply that for any ergodic law~$\BbbP$ on conductances, $\BbbP$-a.s.\ and in $L^1(\BbbP)$,
\begin{equation}
\label{E:4.12}
\frac1{|\Lambda_L|}Q_{\Lambda_L}(t\cdot\psi)\,\underset{L\to\infty}\longrightarrow\,
\E\biggl(\,\,\sum_{x=\hate_1,\dots,\hate_d}a_{0,x}(\omega)\bigl|t\cdot\psi(\omega,x)\bigr|^2\biggr).
\end{equation}
It follows from the construction of the harmonic coordinate that expression on the right coincides with the infimum in \eqref{E:1.5}. (There is no gradient on the right-hand side of \eqref{E:4.12} because $\psi(\omega,0):=0$.) It remains to control the difference on the extreme right of \eqref{E:4.11}. 

Using the quadratic nature of $Q_\Lambda$, the ellipticity assumption \eqref{E:1.3} and Cauchy-Schwarz,
\begin{multline}
\qquad
\frac{\E\bigl|Q_{\Lambda}\bigl(t\cdot\Psi_{\Lambda}\bigr)-Q_{\Lambda}(t\cdot\psi)\bigr|}{|\Lambda|}
\\
\le\frac1\lambda|t|^2\bigl\Vert\nabla(\Psi_{\Lambda}-\psi)\bigr\Vert_{\Lambda,2}^2
+\frac 2\lambda|t|^2\Vert\nabla\psi\Vert_2\bigl\Vert\nabla(\Psi_{\Lambda}-\psi)\bigr\Vert_{\Lambda,2}.\
\qquad
\end{multline}
By Proposition~\ref{prop-approx-corrector-L2} --- which holds for any shift-ergodic (elliptic) law on conductances --- the right-hand side tends to zero as $\Lambda:=\Lambda_L$ increases to~$\Z^d$. Since we know that $|\Lambda_L|^{-1}C^{\,\eff}_L(t)$ is bounded and converges almost surely (e.g., by the Subadditive Ergodic Theorem), it converges also in~$L^1(\BbbP)$. We conclude that the limit value $c_\eff(t)$ is given by \eqref{E:1.5}. 
\end{remark}

\subsection{The Meyers estimate in finite volume}
Key to the proof of Proposition \ref{prop:Meyers estimate} is the Meyers estimate. The term owes its name to Norman G. Meyers~\cite{Meyers} who discovered a bound on $L^p$-continuity (in the right-hand side) of the solutions of Poisson equation with second-order elliptic differential operators in divergence from, provided the associated coefficients are close to a constant. The technical ingredient underpinning this observation is the Calder\'on-Zygmund regularity theory for certain singular integral operators in~$\R^d$. (Incidentally, as noted in~\cite{Meyers}, Meyers' argument is a generalization of earlier work of Boyarskii, cf~\cite[ref.~2 and~3]{Meyers} for systems of first-order PDEs and a version of his result was also derived, though not published, by Calder\'on himself; cf~\cite[page~190]{Meyers}). 

To ease the notation, we will write $\Vert f\Vert_p$ for the canonical norm in $\ell^p(\Lambda)$,
\begin{equation}
\Vert f\Vert_p:=\Bigl(\,\sum_{x\in\Lambda}\bigl|f(x)\bigr|^p\Bigr)^{\ffrac1p},
\end{equation}
throughout the rest of this section. This carries no harm as all of our estimates will be pointwise rather than under expectation.

\smallskip
Let us review the gist of Meyers' argument for functions on~$\Z^d$. Our notation is inspired by that used in Naddaf and Spencer~\cite{Naddaf-Spencer}, who seem to be the first to recognize its significance for the present type of problems, that in Gloria and Otto~\cite{Gloria-Otto}. A general form of the second order difference operator $\LL$ in divergence form is
\begin{equation}
\label{E:div-form}
\LL:=\nabla^\star\cdot A\cdot\nabla,
\end{equation}
where $A=\{A_{ij}(x)\colon i,j=1,\dots,d,\, x\in\Z^d\}$ are $x$-dependent matrix coefficients, $\nabla f(x)$ is a vector whose $i$-th component is $\nabla_if(x):=f(x+\hate_i)-f(x)$ and $\nabla^\star$ is its conjugate acting as $\nabla^\star_i f(x):=f(x)-f(x-\hate_i)$. The above~$\LL$ is explicitly given by
\begin{equation}
\label{E:4.16a}
(\LL f)(x)=\sum_{i,j=1}^d \Bigl(A_{i,j}(x)\bigl[f(x+\hate_i)-f(x)\bigr]-A_{i,j}(x-\hate_j)\bigl[f(x+\hate_i-\hate_j)-f(x-\hate_j)\bigr]\Bigr).
\end{equation}
Now, if $A$ is close to the identity matrix, it makes sense to write
\begin{equation}
\label{E:4.17}
\LL=\Delta+\nabla^\star\cdot (A-\text{id})\cdot\nabla,
\end{equation}
where we noted that the standard lattice Laplacian $\Delta$ corresponds to $\nabla^\star\cdot\text{id}\cdot\nabla$. This formula can be used as a starting point of perturbative arguments.

Consider  a finite set $\Lambda\subset\Z^d$ and let $g\colon\Lambda\cup\partial\Lambda\to\R^d$. Let~$f$ be a solution to the Poisson equation
\begin{equation}
-\LL\,f=\nabla^\star\cdot g,\qquad\text{in }\Lambda,
\end{equation}
with $f:=0$ on~$\partial\Lambda$. Employing \eqref{E:4.17}, we can rewrite this as
\begin{equation}
-\Delta f=\nabla^\star\cdot\bigl[g+(A-\text{id})\cdot\nabla f\bigr].
\end{equation}
The function on the right has vanishing total sum over~$\Lambda$ and hence it lies in the domain of the inverse $(\Delta)^{-1}_\Lambda$ of~$\Delta$ with zero boundary conditions. Taking this inverse followed by one more gradient, and denoting
\begin{equation}
\mathcal{K}_\Lambda:=\nabla(-\Delta)^{-1}_\Lambda\nabla^\star,
\end{equation}
this equation translates to
\begin{equation}
\label{E:4.20}
\nabla f=\mathcal{K}_\Lambda\cdot\bigl[g+(A-\text{id})\nabla f\bigr].
\end{equation}
A first noteworthy point is that this is now an autonomous equation for $\nabla f$. A second point is that, if $\Vert\mathcal{K}_\Lambda\Vert_p$ is the norm of $\mathcal{K}_\Lambda$ as a map (on vector valued functions) $\ell^p(\Lambda)\to\ell^p(\Lambda)$ and $\Vert A-\text{id}\Vert_\infty$ is the least a.s.\ upper bound on the coefficients of $A(x)-\text{id}$, uniform in~$x$, we get
\begin{equation}
\Vert \nabla f\Vert_p\leq \Vert\mathcal{K}_\Lambda\Vert_p\Vert A-\text{id}\Vert_\infty \Vert\nabla f\Vert_p + \Vert\mathcal{K}_\Lambda\Vert_p\Vert g\Vert_p.
\end{equation}
Assuming $\Vert\mathcal{K}_\Lambda\Vert_p\Vert A-\text{id}\Vert_\infty<1$ this yields
\begin{equation}
\label{eqn-meyers-estimate}
\Vert \nabla f\Vert_p\leq \frac{\Vert\mathcal{K}_\Lambda\Vert_p\Vert g\Vert_p}{1-\Vert\mathcal{K}_\Lambda\Vert_p\Vert A-\text{id}\Vert_\infty}.
\end{equation}
Furthermore, the condition $\Vert\mathcal{K}_\Lambda\Vert_p\Vert A-\text{id}\Vert_\infty<1$ ensures  the very existence of a unique solution $\nabla f$ to \eqref{E:4.20} via a contraction argument; \eqref{eqn-meyers-estimate} then implies the continuity of $g\mapsto\nabla f$ in $\ell^p(\Lambda)$.

\smallskip
The aforementioned general facts are relevant for us because $\LL_\omega$ is of the  form \eqref{E:div-form}. Indeed, set $A_{ij}(x):=\delta_{ij}a_{x,x+\hate_i}$ and note that \eqref{E:4.16a} reduces to~\eqref{E:2.8a}. The finite-volume corrector
\begin{equation}
\chi_\Lambda(\omega,x):=\Psi_\Lambda(\omega,x)-x
\end{equation}
then solves the Poisson equation
\begin{equation}
-\LL_\omega\chi_\Lambda=\nabla^\star\cdot g,\quad\text{where}\quad g(x):=(a_{x,x+\hate_1},\dots,a_{x,x+\hate_d}).
\end{equation}
Thanks to \eqref{E:1.3}, this~$g$ is bounded uniformly so, in order to have \eqref{eqn-meyers-estimate} for all finite boxes, our main concern is the following claim:

\begin{theorem}
\label{Lpbound-K}
For each $p\in(1,\infty)$, the operator $\mathcal{K}_{\Lambda_L}$ is bounded in $\ell^p(\Lambda_L)$, uniformly in $L\ge1$.
\end{theorem}

\begin{proofsect}{Proof of Proposition \ref{prop:Meyers estimate} from Theorem~\ref{Lpbound-K}}
Let $p^\ast>4$. Since (in our setting) $\Vert A-\text{id}\Vert_\infty\le\lambda^{-1}-1$, we may choose $\lambda\in(0,1)$ close enough to one so that $\sup_{L\ge1}\Vert\mathcal{K}_{\Lambda_L}\Vert_{p^\ast}\Vert A-\text{id}\Vert_\infty<1$. From the above derivation it follows
\begin{equation}
\label{E:4.23}
\sup_{L\ge1}\Vert \nabla \chi_{\Lambda_L}\Vert_{\Lambda_L,{p^\ast}}<\infty.
\end{equation}
We claim that this implies
\begin{equation}
\label{E:4.24}
\Vert\nabla\chi\Vert_p<\infty,\qquad p<p^\ast.
\end{equation}
Indeed, pick $\alpha>0$ and note that, for any $\epsilon\in(0,\alpha)$,
\begin{equation}
\sum_{x\in\Lambda_L}\1_{\{|\nabla\chi(\cdot,x)|>\alpha\}}
\le
\sum_{x\in\Lambda_L}\1_{\{|\nabla\chi_{\Lambda_L}(\cdot,x)|>\alpha-\epsilon\}}
+
\sum_{x\in\Lambda_L}\1_{\{|\nabla\chi_{\Lambda_L}(\cdot,x)-\nabla\chi(\cdot,x)|>\epsilon\}}.
\end{equation}
Taking expectations and dividing by $|\Lambda_L|$, the left hand side becomes $\BbbP(|\nabla\chi(\cdot,0)|>\alpha)$, while the second sum on the right can be bounded by $\epsilon^{-2}\Vert\nabla\chi_{\Lambda_L}-\nabla\chi\Vert_{\Lambda_L,2}^2$, which tends to zero as $L\to\infty$ by Proposition~\ref{prop-approx-corrector-L2}. Applying Chebyshev's inequality to the first sum on the right and taking $L\to\infty$ followed by $\epsilon\downarrow0$ yields
\begin{equation}
\BbbP\bigl(|\nabla\chi(\cdot,0)|>\alpha\bigr)\le\frac1{\alpha^{p^\ast}}\sup_{L\ge1}\Vert \nabla \chi_{\Lambda_L}\Vert_{\Lambda_L,{p^\ast}}^{p^\ast}.
\end{equation}
Multiplying by $\alpha^{p-1}$ and integrating over $\alpha>0$ then proves \eqref{E:4.24}.
  
Returning to the claims in Proposition~\ref{prop:Meyers estimate}, inequality \eqref{E:4.24} is a restatement of \eqref{eqn-approx-corrector-Lp-1}. Since \twoeqref{E:4.23}{E:4.24} imply the uniform boundedness of $\Vert\nabla(\chi_{\Lambda_L}-\chi)\Vert_{\Lambda_L,p}$, for each $p<p^\ast$, Lemma~\ref{lemma:p_p'} then shows $\Vert\nabla(\chi_{\Lambda_L}-\chi)\Vert_{\Lambda_L,p}\to0$, as $L\to\infty$ for all $p<p^\ast$. This proves \eqref{eqn-approx-corrector-Lp-2} as well.
\end{proofsect}

\subsection{Interpolation}
In the proof of Theorem \ref{Lpbound-K} we will follow the classical argument --- by and large due to Marcinkiewicz --- that is spelled out in Chapter~2 (specifically, proof of Theorem~1 in Section~2.2) of Stein's book~\cite{Stein}. The reasoning requires only straightforward adaptations due to discrete setting and finite volume, but we still prefer to give a full argument to keep the present paper self-contained. A key idea is the use of interpolation between the strong $\ell^2$-type estimate (Lemma~\ref{L2bound-K}) and the weak $\ell^1$-type estimate for~$\KK_{\Lambda_L}$ (Lemma~\ref{weak-1-1-bound-K}). Both of these of course need to hold uniformly in~$L\ge1$. 

\begin{lemma}
\label{L2bound-K}
For any finite $\Lambda\subset\Z^d$, the $\ell^2(\Lambda)$-norm of $\mathcal{K}_{\Lambda}$ satisfies $\Vert \mathcal{K}_{\Lambda}\Vert_2\le1$.
\end{lemma}

\begin{proof}
Let $\HH$ be a Hilbert space and~$T$ a positive self-adjoint, bounded and invertible operator. Then for all $h\in\HH$,
\begin{equation}
\bigl(h,T^{-1}h\bigr)=\sup_{g\in\mathcal \HH}\bigl\{2(g,h)-(g,Tg)\bigr\}.
\end{equation}
We will apply this to $\HH$ given by the space (of $\R$-valued functions)~$\ell^2(\Lambda)$, $T:=\epsilon-\Delta$  and $h:=\nabla^\star\cdot f$ for some $f\colon\Lambda\to\R^d$ with zero boundary conditions outside~$\Lambda$. Then
\begin{equation}
\begin{aligned}
\bigl(\nabla^\star\cdot f,(\varepsilon-\Delta)^{-1}\nabla^\star\cdot f\bigr)&=\sup_{g\in\ell^2(\Lambda)}\bigl\{2(g,\nabla^\star\cdot f)-\varepsilon(g,g)+(g,\Delta g)\bigr\}\\
&=\sup_{g\in\ell^2(\Lambda)}\bigl\{2(\nabla g,f)-\varepsilon(g,g)-(\nabla g,\nabla  g)-(f,f)\bigr\}+(f,f)\\
&=\sup_{g\in\ell^2(\Lambda)}\bigl\{-(\nabla g-f,\nabla g-f)\bigr\}+(f,f)\\
&\leq (f,f),
\end{aligned}
\end{equation}
where we used that $\nabla^\star$ is the adjoint of~$\nabla$ in the space of $\R^d$-valued functions $\ell^2(\Lambda)$ and where the various inner products have to be interpreted either for $\R$-valued or $\R^d$-valued functions accordingly. Taking $\epsilon\downarrow0$, the left-hand side becomes $(f,\mathcal{K}_\Lambda\cdot f)$. The claim follows.
\end{proof}

The second ingredient turns out to be technically more involved.

\begin{lemma}
\label{weak-1-1-bound-K}
$\mathcal{K}_{\Lambda_L}$ is of weak-type (1-1), uniformly in $L>1$. That is, there exists $\widehat K_1$ such that, for all $L>1$, $f\in\ell^1(\Lambda_L)$ and $\alpha>0$,
\begin{equation}\label{weak-type-1-1}
\bigl\vert\{z\in\Lambda_L\colon\vert\mathcal{K}_{\Lambda_L} f(z)\vert>\alpha\}\bigr\vert\leq\widehat K_1\frac{\Vert f\Vert_1}{\alpha}.
\end{equation}
\end{lemma}

Deferring the proof of this lemma to the next subsection, we now show how this enters into the proof of Theorem~\ref{Lpbound-K}.

\begin{proof}[Proof of Theorem~\ref{Lpbound-K} from Lemma~\ref{weak-1-1-bound-K}]
We follow the proof in Stein~\cite[Theorem~5, page~21]{Stein}. We begin with the case $1<p<2$. Let $f\in\ell^p(\Lambda_L)$ and pick $\alpha>0$. Let $f_1:=f\1_{\{\vert f\vert>\alpha\}}$ and $f_2:=f\1_{\{\vert f\vert\leq\alpha\}}$. Then 
\begin{multline}
\qquad
\bigl\vert \{z\in\Lambda_L\colon \vert \mathcal{K}_{\Lambda_L}f(z)\vert>\alpha\}\bigr\vert\leq \bigl\vert \{z\in\Lambda_L\colon \vert \mathcal{K}_{\Lambda_L}f_1\vert>\alpha\}\bigr\vert\\
+\bigl\vert \{z\in\Lambda_L\colon \vert \mathcal{K}_{\Lambda_L}f_2\vert>\alpha\}\bigr\vert. 
\qquad
\end{multline}
Lemmas \ref{L2bound-K} and \ref{weak-1-1-bound-K} then yield
\begin{equation}
\bigl\vert \{z\in\Lambda_L\colon \vert \mathcal{K}_{\Lambda_L}f(z)\vert>\alpha\}\bigr\vert\leq \widehat K_1\frac{\Vert f_1\Vert_1}{\alpha} +\widehat K_2\frac{\Vert f_2\Vert_2^2}{\alpha^2},
\end{equation}
with $\widehat K_1$ and $\widehat K_2$ independent of $L$. Multiplying by $\alpha^{p-1}$ and integrating, we infer
\begin{equation}
\begin{aligned}
\Vert \mathcal{K}_{\Lambda_L}f\Vert_p^p&=p\int_0^\infty\alpha^{p-1}\bigl\vert \{z\in\Lambda_L\colon \vert \mathcal{K}_{\Lambda_L}f(z)\vert>\alpha\}\bigr\vert\,\textd \alpha\\
&\leq p\sum_z\int_0^\infty \Bigl(\widehat K_1\alpha^{p-2}\vert f(z) \vert\1_{\{\vert f\vert>\alpha\}} +\widehat K_2\alpha^{p-3}\vert f(z) \vert^2\1_{\{\vert f\vert\leq\alpha\}}\Bigr)\,\textd \alpha\\
&=p\widehat K_1\sum_z\vert f(z) \vert\int_0^{\vert f(z) \vert} \alpha^{p-2}\,\textd \alpha+p\widehat K_2\sum_z\vert f(z) \vert^2\int_{\vert f(z) \vert}^\infty \alpha^{p-3}\,\textd \alpha\\
&=\frac{p\widehat K_1}{p-1}\sum_z\vert f(z) \vert^p+\frac{p\widehat K_2}{2-p}\sum_z\vert f(z) \vert^p,
\end{aligned}
\end{equation}
proving the assertion in the case $1<p<2$. 

For $p\in(2,\infty)$, the facts that $\KK_{\Lambda}$ is obviously symmetric and that the norm admits the representation 
\begin{equation}
\Vert\KK_{\Lambda}\Vert_p =\sup_{\Vert f \Vert_p=1}\sup_{\Vert g \Vert_q=1}\frac 1 {|\Lambda|}\sum_{x\in\Lambda}(\KK_\Lambda f)(x)g(x),
\end{equation} for~$q$ equal to the index dual to~$p$, imply that $\Vert\KK_{\Lambda}\Vert_p=\Vert\KK_{\Lambda}\Vert_q$. Hence $\sup_{L\ge1}\Vert\KK_{\Lambda_L}\Vert_p<\infty$ for all $p\in(1,\infty)$.
\end{proof}

\subsection{Weak type-(1,1) estimate}
It remains to prove Lemma \ref{weak-1-1-bound-K}. The strategy is to represent the operator using a singular kernel that has a ``nearly $\ell^1$-integrable'' decay. Let $G_\Lambda(x,y)$ be the Green function (i.e., inverse) of the Laplacian $\Delta$ on $\Lambda$ with zero boundary condition on~$\partial\Lambda$.

\begin{lemma}
The operator $\mathcal{K}_\Lambda$ admits the representation 
\begin{equation}
\hate_i\cdot\bigl[\KK_\Lambda\cdot f(x)\bigr]=\sum_{y\in\Lambda}\sum_{j=1}^d \bigl[\nabla_i^{(1)}\nabla_j^{(2)}G_\Lambda(x,y)\bigr]f_j(y),
\end{equation}
where the superscripts on the $\nabla$'s indicate which of the two variables the operator is acting on.
\end{lemma}

\begin{proof}
Since both $G_\Lambda$ and $f$ vanish outside $\Lambda$, we have
\begin{equation}
\begin{aligned}
\hate_i\cdot\bigl[\KK_\Lambda\cdot f(x)\bigr]=&\nabla_i\Bigl(\,\sum_{y\in\Lambda}G_\Lambda(\cdot,y)\bigl(\nabla^\star\cdot f\bigr)(y)\Big)(x)\\
=&\sum_{y\in\Z^d}\Bigl(\bigl(G_\Lambda(x+\hate_i,y)-G_\Lambda(x,y)\bigr)\sum_{j=1}^d[f_j(y-\hate_j)-f_j(y)]\Bigr)\\
=&\sum_{j=1}^d\sum_{y\in\Z^d}\bigl(G_\Lambda(x+\hate_i,y+\hate_j)-G_\Lambda(x,y+\hate_j)\bigr)f_j(y)\\
&\qquad\qquad\quad-\sum_{j=1}^d\sum_{y\in\Z^d}\bigl(G_\Lambda(x+\hate_i,y)-G_\Lambda(x,y)\bigr)f_j(y).
\end{aligned}
\end{equation}
This is exactly the claimed expression.
\end{proof}

Crucial for the proof of the weak-type (1,1)-estimate in Lemma~\ref{weak-1-1-bound-K} is an integrable decay estimate on the gradient of the kernel of the operator $\KK_\Lambda$: 

\begin{proposition}
\label{green-decay-box}
There exists $C>0$ independent of $L$ such that
\begin{equation}
\label{E:4.35}
\bigl\vert\nabla_i^{(2)}\nabla_j^{(1)}\nabla_k^{(2)}G_{\Lambda_L}(x,y)\bigr\vert\leq  \frac{C}{\vert x-y\vert^{d+1}}
\end{equation}
for all $x,y\in\Lambda_L$ and $i,j,k\in\{1,\ldots,d\}$. 
\end{proposition}

Although \eqref{E:4.35} is certainly not unexpected, and perhaps even well-known, we could not find an exact reference and therefore provide an independent proof in Section \ref{sec-Greens-function-decay}. With this estimate in hand, we can now turn to the proof of  Lemma \ref{weak-1-1-bound-K}.

\begin{proof}[Proof of Lemma \ref{weak-1-1-bound-K} from Proposition~\ref{green-decay-box}]
To ease the notation, we will write $\Lambda:=\Lambda_L$ (note that all bounds will be uniform in~$L$) and, resorting to components, write $\KK_\Lambda$ for the scalar-to-scalar operator with kernel $\KK_\Lambda^{(i,j)}(x,y):=\nabla_i^{(1)}\nabla_j^{(2)}G_\Lambda(x,y)$ for some fixed $i,j\in\{1,\ldots,d\}$. For the most part, we adapt the arguments in Stein~\cite[pages~30-33]{Stein}.

Given a function $f\colon\Lambda\to\R$, regard it as extended by zero outside $\Lambda$. Pick $\alpha>0$ and consider a partition of~$\Z^d$ into cubes of side $3^r$, where $r$ is chosen so large that $3^{-rd}\Vert f\Vert_1\leq\alpha$. Naturally, each cube in the partition further divides into~$3^d$ equal-sized sub-cubes of side~$3^{r-1}$, which subdivide further into sub-cubes of side $3^{r-2}$, etc. We will now designate these to be either \emph{good cubes} or \emph{bad cubes} according to the following recipe. All cubes of side $3^r$ are \emph{ex definitio} good. With $Q$ being one of these sub-cubes of side $3^{r-1}$, we call $Q$ good if
\begin{equation}
\label{E:4.36}
\frac 1{\vert Q\vert}\sum_{z\in Q}\bigl|f(z)\bigr|\leq\alpha,
\end{equation}
and bad otherwise. For each good cube, we repeat the process of partitioning it into $3^d$ equal-size sub-cubes and designating each of them to be either good or bad depending on whether \eqref{E:4.36} holds or not, respectively. The bad cubes are not subdivided further.

Iterating this process, we obtain a finite set $\BB$ of bad cubes which covers the (bounded) region $B:=\bigcup_{Q\in\BB} Q$. We define $G:=\Z^d\setminus B$, the good region, and note that 
\begin{equation}
\label{good-sets}
\bigl| f(z)\bigr|\leq \alpha,\qquad z\in G,
\end{equation}
and 
\begin{equation}\label{bad-sets}
\alpha < \frac 1{\vert Q\vert}\sum_{z\in Q}\bigl|f(z)\bigr|\leq 3^d\alpha,\qquad Q\in\BB,
\end{equation}
where the last inequality is due to the fact that the parent cube of a bad cube is good. Next we define the ``good'' function
\begin{equation}
\label{E:4.39}
g(z):=\begin{cases}
f(z),\qquad &z\in G\\
		\frac 1{\vert Q\vert}\sum_{z\in Q}f(z),\qquad &z\in Q\in\BB.
\end{cases}
\end{equation}
The ``bad'' function, defined by $b:=f-g$, then satisfies
\begin{equation}
\label{E:4.40}
\begin{alignedat}{3}
b(z)&=0,\qquad &\quad&z\in G,\\
\sum_{z\in Q}b(z)&=0,\qquad &&Q\in\BB.
\end{alignedat}
\end{equation}
Since $\mathcal{K}_{\Lambda}f=\mathcal{K}_{\Lambda}g+\mathcal{K}_{\Lambda}b$,  as soon as
\begin{equation}
\label{E:4.41}
\bigl\vert\{z:\vert\mathcal{K}_{\Lambda} g(z)\vert>\ffrac{\alpha}{2}\}\bigr\vert\leq\frac{\widehat K_1\Vert f\Vert_1}{2\alpha}\quad\text{ AND }\quad\bigl\vert\{z:\vert\mathcal{K}_{\Lambda} b(z)\vert>\ffrac{\alpha}{2}\}\bigr\vert\leq\frac{\widehat K_1\Vert f\Vert_1}{2\alpha},
\end{equation}
the desired bound \eqref{weak-type-1-1} will hold. We will now show these bounds in separate arguments.

Considering $g$ first, we note that $\Vert g\Vert_2^2$ is bounded by a constant times $\alpha\Vert f\Vert_1$. Indeed, for $z\in B$ let $Q_z$ denote the bad cube containing $z$. Then
\begin{equation}
\begin{aligned}
\sum_{z\in \Z^d}g(z)^2&=\sum_{z\in G} f(z)^2+\sum_{z\in B}g(z)^2\\
&\leq\alpha\sum_{z\in G}\bigl|f(z)\bigr|+\sum_{z\in B}\Big(\frac 1{\vert Q_z\vert}\sum_{y\in Q_z}f(z)\Big)^2\\
&\leq\alpha\Vert f\Vert_1+3^d\alpha\sum_{z\in B}\frac 1{\vert Q_z\vert}\sum_{y\in Q_z}\bigl|f(z)\bigr|\\
&\leq (3^d+1)\alpha\Vert f\Vert_1
\end{aligned}
\end{equation}
by using \eqref{good-sets} on $G$ and \eqref{bad-sets} on $B$. By Chebychev's inequality and Lemma \ref{L2bound-K},
\begin{equation}\label{wt11-1}
\big\vert\{z:\vert\mathcal{K}_{\Lambda} g(z)\vert>\alpha\}\big\vert\leq\frac{\Vert \mathcal{K}_\Lambda g\Vert_2^2}{\alpha^2}\leq\frac{(3^d+1)\Vert\mathcal{K}_\Lambda\Vert_2^2\,\Vert f\Vert_1}\alpha.
\end{equation}
Note that this yields an estimate that is uniform in $\Lambda:=\Lambda_L$ because $\Vert\mathcal{K}_\Lambda\Vert_2\le1$ by Lemma~\ref{L2bound-K}.

Let us turn to the estimate in \eqref{E:4.41} concerning~$b$. Let $\{Q_k\colon k=1,\ldots,|\BB|\}$ be an enumeration of the bad cubes and let $b_k:=b\1_{Q_k}$ be the restriction of~$b$ onto~$Q_k$. Abusing the notation to the point where we write $\KK_\Lambda(x,y)$ for the kernel governing $\KK_\Lambda$, from \eqref{E:4.40} we then have
\begin{equation}
\label{E:4.44}
\KK_\Lambda b_k(z)=\sum_{y\in Q_k}\bigl[\KK_\Lambda(z,y)-\KK_\Lambda(z,y_k)\bigr]b(y),
\end{equation}
where $y_k$ is the center of~$Q_k$ (remember that all cubes are odd-sized).
Let $\tilde Q_k$ denote the cube centered at $y_k$ but of three-times the size --- i.e., $\tilde Q_k$ is the union of $Q_k$ with the adjacent $3^d-1$ cubes of the same side. The bound now proceeds depending on whether $z\in\tilde Q_k$ or not.

For $z\not\in\tilde Q_k$, the distance between $z$ and any~$y\in Q_k$ is proportional to the distance between~$z$ and~$y_k$. Proposition~\ref{green-decay-box} thus implies
\begin{equation}
\bigl|\KK_\Lambda(z,y)-\KK_\Lambda(z,y_k)\bigr|\le C\,\frac{\text{diam}(Q_k)}{\,\vert z-y_k\vert^{d+1}},
\qquad z\not\in\tilde Q_k.
\end{equation}
Moreover, thanks to \eqref{E:4.39}, 
\begin{equation}
\sum_{y\in Q_k}\vert b(y)\vert\leq\sum_{y\in Q_k}\bigl(\vert f(y)\vert+\vert g(y)\vert\bigr)\leq 2\sum_{y\in Q_k}\vert f(y)\vert.
\end{equation}
Using these in \eqref{E:4.44} yields
\begin{equation}
\vert\mathcal{K}_\Lambda b_k(z)\vert\leq C\,\frac{\text{diam}(Q_k)}{\,\vert z-y_k\vert^{d+1}}\sum_{y\in Q_k}\vert f(y)\vert.
\end{equation}
Summing over all $z\not\in\tilde Q_k$ and taking into account that $\vert z-y_k\vert\geq \text{diam}(Q_k)$ for $z\in\tilde Q_k$, we conclude
\begin{equation}
\begin{aligned}
\sum_{z\in\Lambda\setminus\tilde Q_k}\vert\mathcal{K}_\Lambda b_k(z)\vert&\leq C\,\text{diam}(Q_k)\sum_{y\in Q_k}\vert f(y)\vert\sum_{z\colon\vert z-y_k\vert\geq\text{diam}(Q_k)}\frac {1}{\vert z-y_k\vert^{d+1}}\\
&\leq \tilde C\sum_{y\in Q_k}\vert f(y)\vert
\end{aligned}
\end{equation}
for some constant $\tilde C$. Setting $\tilde B:=\bigcup_k\tilde Q_k$ and summing over $k$, we obtain
\begin{equation}
\sum_{z\in\Lambda\setminus\tilde B}\vert\mathcal{K}_\Lambda b(z)\vert\leq\tilde C\sum_{y\in B}\vert f(y)\vert\leq\tilde C \Vert f\Vert_1,
\end{equation}
which by an application of Markov's inequality yields
\begin{equation}\label{wt11-2}
\bigl\vert\{z\in\Lambda\setminus\tilde B\colon \vert\mathcal{K}_\Lambda b(z)\vert\geq\alpha\}\bigr\vert \leq \frac{\tilde C\Vert f\Vert_1}{\alpha}.
\end{equation}
i.e., a bound of the desired form.

To finish the proof, we still need to take care of $z\in\tilde B$. Here we get (and this is the only step where we are forced to settle on \emph{weak}-type estimates),
\begin{equation}
\label{E:4.49}
\begin{aligned}
\bigl\vert\{z\in\tilde B\colon \vert\mathcal{K}_\Lambda b(z)\vert\geq\alpha\}\bigr\vert&\leq\vert\tilde B\vert\leq 3^d\sum_k\vert Q_k\vert\\
&\leq 3^d\sum_k\frac 1\alpha\sum_{z\in Q_k}\bigl\vert f(z)\bigr\vert\leq \frac{3^d\Vert f\Vert_1}\alpha.
\end{aligned}
\end{equation}
The bound \eqref{weak-type-1-1} then follows by combining \eqref{wt11-1}, \eqref{wt11-2} and \eqref{E:4.49}.
\end{proof}

\subsection{Triple gradient of finite-volume Green's function}
\label{sec-Greens-function-decay}\noindent
In order to finish the proof of Theorem~\ref{Lpbound-K}, we still need to establish the decay estimate in Proposition~\ref{green-decay-box}. This will be done by invoking a corresponding bound in the full lattice and reducing it onto a box by reflection arguments. (This is the sole reason why we restrict to rectangular boxes; more general domains require considerably more sophisticated methods.) 

For~$\varepsilon>0$, let $G^\varepsilon$ denote the Green function associated with the discrete Laplacian $\Delta$ on $\Z^d$ with killing rate $\varepsilon>0$, i.e., $G^\varepsilon(\cdot,\cdot)$ is the kernel of the bounded operator $(\varepsilon-\Delta)^{-1}$ on $\ell^2(\Z^d)$. This function admits the probabilistic representation
\begin{equation}
G^\varepsilon(x,y)=\sum_{k=0}^\infty \frac{P^x\bigl(X_k=y\bigr)}{(1+\varepsilon)^{k+1}},
\end{equation}
where~$X$ is the simple random walk and~$P^x$ is the law of~$X$ started at~$x$. This function depends only on the difference of its arguments, so we will interchangeably write $G^\varepsilon(x,y)=G^\varepsilon(x-y)$. We now claim:

\begin{lemma}
\label{green-eps-decay-lattice}
There exists $\widehat C>0$ such that, for all $\varepsilon>0$, all $i,j,k\in\{1,\ldots d\}$ and all $x\ne 0$,
 \begin{equation}
\bigl\vert\nabla_i\nabla_j\nabla_kG^\varepsilon(x)\bigr\vert\leq \frac{\widehat C}{\,\,\vert x\vert^{d+1}}.
\end{equation} 
\end{lemma}

\begin{proof}[Sketch of proof]
This is a mere extension (by adding one more gradient) of the estimates from in Lawler~\cite[Theorem~1.5.5]{Lawler}. (Strictly speaking, this theorem is only for the transient dimensions but, thanks to $\varepsilon>0$, the same proofs would apply here.) The main idea is to use translation invariance of the simple random walk to write $G^\epsilon(x)$ as a Fourier integral and then control the gradients thereof under the integral sign.  We leave the details as an exercise to the reader.
\end{proof}

We now state and prove a stronger form of Proposition~\ref{green-decay-box}.

\begin{lemma}
\label{green-eps-decay-box}
There exists $C>0$ such that, for all $L>1$, $\varepsilon>0$ and arbitrary $i,j,k\in\{1,\ldots d\}$, 
\begin{equation}
\vert\nabla_i^{(2)}\nabla_j^{(1)}\nabla_k^{(2)}G^\varepsilon_{\Lambda_L}(x,y)\vert\leq  \frac{C}{\vert x-y\vert^{d+1}}
\end{equation}
for all $x,y\in\Lambda$ and all $i,j,k\in\{1,\ldots,d\}$. 
Here, the superscripts on the operators indicate the variable the operator is acting on.
\end{lemma}

\begin{proof}
Throughout, we fix $L\in\N$ and denote $\Lambda:=\Lambda_L$.
The proof is based on the Reflection Principle for the simple random walk on~$\Z^d$. Let $X^{(i)}$ denote the $i$-th component of $X$ and let 
\begin{equation}
\tau^i_0:=\inf\{k\ge0\colon X^{(i)}_k=0\}\quad\text{and}\quad 
\tau^i_L:=\inf\{k\ge0\colon X^{(i)}_k=L\}.
\end{equation} 
For $y\in\Lambda_i$ with components $y=(y_1,\ldots,y_d)$, and integer-valued indices $n\in\Z$, put
\begin{equation}
\begin{aligned}
r^i_{2n}(y)&:=(y_1,\ldots,2nL+y_i,\ldots,y_d)\\
r^i_{2n+1}(y)&:=(y_1,\ldots,2(n+1)L-y_i,\ldots,y_d).
\end{aligned}
\end{equation}
Our first claim is that, for $i\in\{1,\ldots,d\}$,
\begin{equation}
\label{eqn:reflection-prob-coordinate}
P^x\bigl(X_k=y,\tau^i_0> k,\tau_L^i>k\bigr)=\sum_{n\in\Z}(-1)^n P^x\bigl(X_k=r^i_n(y)\bigr).
\end{equation}
First we note that, the restriction on the length of~$X$ makes this effectively a finite sum so we only have to exhibit appropriate cancellations due to the alternating sign. Let $A_m^k$, for $k,m\in\N$ and $n\in\Z$, denote the set of paths of length~$k$ starting at~$x$, ending in~$\{r^i_n(y)\colon n\in\Z\}$ and visiting $\{x\in\Z^d\colon x_i=L\Z\}$ exactly~$m$ times. Let $s(p):=0$ if the path~$p\in A_m^k$ ends in $\{r^i_{2n}(y)\colon n\in\Z\}$ and $s(p):=1$ otherwise. The cancellation will arise from the fact that, 
\begin{equation}
\label{eqn:reflection-alternate}
\sum_{ p \in A^k_m } (-1)^{ s(p) } = 0,\qquad m>0.
\end{equation}
To see this consider a map of $A^k_m$, $m>0$, onto itself by taking a path and reflecting the part after the last visit to $\{x\in\Z^d\colon x_i=L\Z\}$. This is a bijection from $A^k_m$ onto itself which changes the sign of $(-1)^{s(p)}$. Since $A_m^k$ is finite, the sum must vanish.

As all paths in $A^k_m$ have the same probability, we are permitted to multiply \eqref{eqn:reflection-alternate} by the probability of each respective path and get for all $m>0$ that
\begin{equation}
\label{eqn:reflection-alternate2}
\begin{aligned}
0=\sum_{ p \in A^k_m } (-1)^{ s(p) } P^x(X_{[k]}=p)
&=\sum_{ p \in A^k_m,\,n\in\Z } (-1)^{ s(p) } P^x\bigl(X_{[k]}=p,X_k=r^i_n(y)\bigr)\\
&=\sum_{ p \in A^k_m,\,n\in\Z } (-1)^n P^x\bigl(X_{[k]}=p,X_k=r^i_n(y)\bigr)\\
&=\sum_{n\in\Z } (-1)^n P^x\bigl(X_{[k]}\in A^k_m,X_k=r^i_n(y)\bigr),
\end{aligned}
\end{equation}
where $X_{[k]}$ denotes the path of the random walk up to time $k$. We now verify \eqref{eqn:reflection-prob-coordinate} by
\begin{equation}
\begin{aligned}
P^x\bigl(X_k=y,\tau^i_0> k,\tau_L^i>k\bigr)&=
P^x\bigl(X_{[k]}\in A^k_0\bigr)\\
&=\sum_{n\in\Z } (-1)^n P^x\bigl(X_{[k]}\in A^k_0,X_k=r^i_n(y)\bigr)\\
&=\sum_{\,n\in\Z }\,\sum_{m\in\N} (-1)^n P^x\bigl(X_{[k]}\in A^k_m,X_k=r^i_n(y)\bigr)\\
&=\sum_{n\in\Z } (-1)^n P^x\bigl(X_k=r^i_n(y)\bigr),
\end{aligned}
\end{equation}
where \eqref{eqn:reflection-alternate2} was used in the third equality. (There are no convergence issues as all sums remain effectively finite because~$k$ is fixed.)
This obviously holds regardless of any restriction of the other components of the walk, which means that we have in particular
\begin{multline}
\qquad
\label{eqn:reflection-prob-coordinate2}
P^x\bigl(X_k=y,\,\tau^j_0>k,\tau^j_L>k\,\,\forall j>i\bigr)
\\=\sum_{n\in\Z}(-1)^n P^x\bigl(X_k=r^{i+1}_n(y),\,\tau^j_0>k,\tau^j_L>k\,\,\forall j>i+1\bigr)\qquad
\end{multline}
for each $i\in\{0,\ldots,d-1\}$. 

We are now ready to establish the desired representation for the Green function. The argument proceeds by induction on dimension. Abusing our earlier notation, denote
\begin{equation}
\begin{aligned}
\Lambda_0&:=\Lambda_L=\{0,\ldots,L\}^d,\\
\Lambda_{i}&:=\Z^i\times\{0,\ldots,L\}^{d-i},\qquad i=1,\ldots,d-1,\\
\Lambda_d&:=\Z^d,
\end{aligned}
\end{equation}
For any $i\in\{0,\ldots,d\}$, the Green function~$G^\varepsilon_{\Lambda_i}$ on $\Lambda_i$ with zero boundary condition is given by
\begin{equation}
G^\varepsilon_{\Lambda_i}(x,y)=\sum_{k=0}^\infty (1+\epsilon)^{-k-1}P^x\bigl(X_k=y,\,\tau^j_0>k,\tau^j_L>k\,\,\forall j> i\bigr).
\end{equation}
Applying \eqref{eqn:reflection-prob-coordinate2} to every probability term, we obtain for each $i\in\{0,\ldots,d-1\}$
\begin{equation}\label{eqn:reflection-green-coordinate2}
G^\varepsilon_{\Lambda_i}(x,y)=\sum_{n\in\Z}(-1)^n G^\varepsilon_{\Lambda_{i+1}}(x,r^{i+1}_n(y)).
\end{equation}
Consecutive applications of this equality yield
\begin{equation}\label{eqn:reflection-green-coordinate2}
G^\varepsilon_{\Lambda}(x,y)=\sum_{z\in\Z^d}(-1)^{z_1+\ldots+z_d} G^\varepsilon_{\Z^d}(x,r_z(y))
\end{equation}
for all $x,y\in\Lambda$, where we abbreviate $r_z=r^1_{z_1}\circ\cdots\circ r^d_{z_d}$. From Lemma \ref{green-eps-decay-lattice}, we thus obtain
\begin{equation}
\begin{aligned}
\bigl\vert\nabla_i^{(2)}\nabla_j^{(1)}\nabla_k^{(2)}G_\Lambda^\varepsilon(x,y)\bigr\vert
\leq\sum_{z\in\Z^d}\bigl\vert\nabla_i^\star\nabla_j\nabla_k^\star G^\varepsilon(x-r_z(y))\bigr\vert
\leq\sum_{z\in\Z^d}\frac{\widehat C}{\,\vert\, x-r_z(y)\vert^{d+1}}
\end{aligned}
\end{equation}
for all $x,y\in\Lambda$. Let $x,y\in\Lambda$ and abbreviate $z_\text{max}=\max_{i=1}^d\vert z_i\vert$. Whenever $z_\text{max}\leq 1$, we have $\vert\, x-r_z(y)\vert\geq\vert\, x-y\vert$ as reflection always increases the distance between points in $\Lambda$. If $z_\text{max}>1$, we may even estimate $\vert\, x-r_z(y)\vert\geq d^{-1/2}L\vert\, z\vert\geq d^{-1}\vert\, x-y\vert\vert\, z\vert.$ The latter is verified quickly using $d^{1/2}z_\text{max}\geq\vert z\vert\geq z_\text{max}$ and the fact that $z_\text{max}$ is at least $2$ in this case. Therefore, we obtain
\begin{equation}
\begin{aligned}
\bigl\vert\nabla_i^{(2)}\nabla_j^{(1)}\nabla_k^{(2)}G_\Lambda^\varepsilon(x,y)\bigr\vert
&\leq\sum_{z\,\colon z_\text{max}\leq 1}\frac{\widehat C}{\,\vert\, x-y\vert^{d+1}}+\sum_{z\,\colon z_\text{max}>1}\frac{d^{d+1}\widehat C}{\,\vert\, x-y\vert^{d+1}\vert\, z\vert^{d+1}}\\
&\leq\frac{\widehat C}{\,\vert\, x-y\vert^{d+1}}\Big(3^d+d^{d+1}\sum_{z\neq 0}\frac 1{\,\vert\, z\vert^{d+1}}\Big),
\end{aligned}
\end{equation}
which is the desired estimate.
\end{proof}

We are now ready to complete the proof of Theorem~\ref{Lpbound-K}:

\begin{proof}[Proof of Proposition~\ref{green-decay-box}]
Although the $\varepsilon\downarrow0$ limit of $G^\epsilon$ exists only in $d\ge3$, for gradients we have $\nabla G(x,y)=\lim_{\varepsilon\downarrow0}\nabla G^\epsilon(x,y)$ in all $d\ge1$. Since the bound in Lemma~\ref{green-eps-decay-box} holds uniformly in~$\varepsilon>0$, we get the claim in all $d\ge1$.
\end{proof}

\section{Perturbed harmonic coordinate}
\label{sec-perturbed-coordinate}
\noindent
In this section we will prove Propositions~\ref{prop-perturbed-corrector} and~\ref{prop-2.8}. 
Abandoning our earlier notation, let
\begin{equation}
G_\Lambda(x,y;\omega)=(-\LL_\omega)^{-1}(x,y)
\end{equation}
denote the Green function in $\Lambda$ with Dirichlet boundary condition for conductance configuration~$\omega$. (Thus, the simple-random walk Green function from Section~\ref{sec-Meyers} corresponds to $\omega:=1$.) The Green function is the fundamental solution to the Poisson equation, i.e.,
\begin{equation}
\label{E:5.2}
\begin{cases}
-\LL_\omega G_\Lambda(x,z,\omega)=\delta_x(z)\qquad&\text{if }z\in\Lambda,
\\*[1mm]
\,\,\,G_\Lambda(x,z,\omega)=0,\qquad&\text{if }z\in\partial\Lambda,
\end{cases}
\end{equation}
where $\delta_x(z)$ is the Kronecker delta. 
Note that $G_\Lambda$ is defined for all $\omega\in\Omega$. The solution to \eqref{E:5.2} is naturally symmetric,
\begin{equation}
\label{E:5.3}
G_\Lambda(x,y;\omega)=G_\Lambda(y,x;\omega),\qquad x,y\in\Lambda,
\end{equation}
and so we can extend it to a function on $\Lambda\cup\partial\Lambda$ by setting $G_\Lambda(x,\cdot;\omega)=0$ whenever $x\in\partial\Lambda$.
Here is a generalized form of the representation~\eqref{E:2.21} (we thank a careful referee who pointed out that this result has appeared in a very similar form in \cite{Gloria-Otto}, Lemma 2.4):

\begin{lemma}[Rank-one perturbation]
\label{lemma-5.1}
For a finite $\Lambda\subset\Z^d$ let $x,y\in\Lambda$ be nearest neighbors. For any $\omega,\omega'$ such that $\omega'_b=\omega_b$ except at $b:=\langle x,y\rangle$, and any $z\in\Lambda\cup\partial\Lambda$,
\begin{multline}
\qquad
\Psi_\Lambda(\omega',z)-\Psi_\Lambda(\omega,z)
\\=-(\omega_{xy}'-\omega_{xy})\bigl[G_\Lambda(z,y;\omega')-G_\Lambda(z,x;\omega')\bigl]\bigl[\Psi_\Lambda(\omega,y)-\Psi_\Lambda(\omega,x)\bigr].
\quad
\end{multline}
\end{lemma}

\begin{proofsect}{Proof}
Suppose $\omega,\omega'\in\Omega$ are such that $\omega'$ equals $\omega$ except at the edge $b:=\langle x,y\rangle$, where $\omega_b':=\omega_b+\epsilon$. Define the function $\Phi_\Lambda\colon\Lambda\cup\partial\Lambda\to\R^d$ by
\begin{equation}
\Phi_\Lambda(z):=\Psi_\Lambda(\omega,z)-\epsilon\bigl[G_\Lambda(z,y;\omega')-G_\Lambda(z,x;\omega')\bigl]\bigl[\Psi_\Lambda(\omega,y)-\Psi_\Lambda(\omega,x)\bigr].
\end{equation}
We claim that
\begin{equation}
\label{E:5.6}
\LL_{\omega'}\Phi_\Lambda=0\qquad\text{in }\Lambda.
\end{equation}
Since for $z\in\partial\Lambda$ we have $\Phi_\Lambda(z)=\Psi_\Lambda(\omega,z)=z$, this will imply $\Phi_\Lambda(\cdot)=\Psi_\Lambda(\omega',\cdot)$ thanks to the uniqueness of the solution of the Dirichlet problem.

In order to show \eqref{E:5.6}, we first use \twoeqref{E:5.2}{E:5.3} to get
\begin{equation}
\LL_{\omega'}\Phi_\Lambda(z)=\LL_{\omega'}\Psi_\Lambda(\omega,z)-\epsilon\bigl[\delta_y(z)-\delta_x(z)\bigr]\bigl[\Psi_\Lambda(\omega,y)-\Psi_\Lambda(\omega,x)\bigr].
\end{equation}
To deal with the term $\LL_{\omega'}\Psi_\Lambda(\omega,z)$, we think of of $\LL_{\omega'}$ as a matrix of dimension $|\Lambda|$. For its coefficients $\LL_{\omega}(z,z'):=\langle\delta_z,\LL_\omega\delta_{z'}\rangle_{\ell^2(\Lambda)}$ we obtain
\begin{equation}
\label{E:5.8rr}
\LL_{\omega'}(z,z')=\LL_\omega(z,z')+\epsilon\bigl[\delta_y(z)-\delta_x(z)\bigr]\bigl[\delta_y(z')-\delta_x(z')\bigr].
\end{equation}
Using that $\LL_\omega\Psi_\Lambda(\omega,z)=0$ for $z\in\Lambda$, we now readily confirm \eqref{E:5.6}.
\end{proofsect}

\begin{proofsect}{Proof of Proposition~\ref{prop-perturbed-corrector}}
Set $y:=x+\hate_i$ and denote $\nabla_i f(z):=f(z+\hate_i)-f(z)$. Lemma~\ref{lemma-5.1} shows
\begin{equation}
\label{E:5.8d}
\nabla_i\Psi_\Lambda(\omega',x)
=\Bigl[1-(\omega_b'-\omega_b)\nabla_i^{(1)}\nabla_i^{(2)}G_\Lambda(x,x,\omega')\Bigr]\nabla_i\Psi_\Lambda(\omega,x),
\end{equation}
where the superindices on $\nabla$ indicate which variable is the operator acting on. To prove the claim we need to show
\begin{equation}
\label{E:5.9}
\bigl[\nabla_i^{(1)}\nabla_i^{(2)}G_\Lambda(x,x,\omega)\bigr]^{-1}=\inf\bigl\{Q_{\Lambda}(f)\colon f(y)-f(x)=1,\,0\le f\le 1,\,f_{\partial\Lambda}=0\bigr\},
\end{equation}
where the conductances in $Q_\Lambda$ correspond to~$\omega$. For this, let~$f$ be the minimizer of the right-hand side. The method of Largrange multipliers shows
\begin{equation}
\label{E:5.12}
-\LL_\omega f(z)=\alpha\bigl[\delta_y(z)-\delta_x(z)\bigr].
\end{equation}
Thanks to \eqref{E:5.2}, this is solved by
\begin{equation}
f(z)=\alpha\bigl[G_\Lambda(y,z;\omega)-G_\Lambda(x,z;\omega)\bigr]=\alpha\nabla_i^{(1)}G_\Lambda(x,z;\omega)
\end{equation}
which in light of the constraint $f(y)-f(x)=1$ gives $\alpha=[\nabla_i^{(1)}\nabla_i^{(2)}G_\Lambda(x,x,\omega)]^{-1}$.
Since also $Q_\Lambda(f)=\langle f,-\LL_\omega f\rangle_{\ell^2(\Lambda)}$, \eqref{E:5.12} gives $Q_\Lambda(f)=\alpha$ and so \eqref{E:5.9} holds. The correspondence \eqref{E:2.21} then follows from \twoeqref{E:5.8d}{E:5.9}.

It remains to prove the equalities in \eqref{E:2.24b}. The first equality follows form \eqref{E:5.8d} by exchanging the roles of~$\omega$ and~$\omega'$. For the second equality, we note that \eqref{E:5.8rr} relates the said factor to the ratio of double gradients of the Green function at~$\omega$ and~$\omega'$ which we abbreviate as~$\gg^{(i)}_\Lambda(\omega,x)$. 
\end{proofsect}


Finally, it remains to establish the limit \eqref{E:2.22}, including all of its stated properties:

\begin{proofsect}{Proof of Proposition~\ref{prop-2.8}}
Thanks to ellipticity  restriction \eqref{E:1.3}, we have a bound on this quantity in terms of the lattice Laplacian. Keeping in mind the representation in \eqref{E:5.9}, we have in fact that, for some $c=c(\lambda)\in(0,1)$ related to the double gradient of the Green function for all conductances equal to one,
\begin{equation}
c<\nabla_i^{(1)}\nabla_i^{(2)}G_\Lambda(x,x,\omega')<1/c
\end{equation}
uniformly in $\Lambda$. Moreover, again by \eqref{E:5.9}, $\Lambda\mapsto\nabla_i^{(1)}\nabla_i^{(2)}G_\Lambda(x,x,\omega')$ is non-decreasing in $\Lambda$ and so the limit exists. The formula \eqref{E:2.25a} and the claimed stationarity then follow.
\end{proofsect}

\noindent

\section*{Acknowledgments}
\noindent
This research was partially supported by the NSF grant DMS-1106850, NSA grant H98230-11-1-0171 and the GA\v CR project P201/11/1558. Two of the authors (M.S.\ and T.W.) gratefully acknowledge financial support from the ESF Grant ``Random Geometry of Large Interacting Systems and Statistical Physics'' (RGLIS) that made their stay at UCLA, and the work presented in this note, possible. The authors also wish to thank three anonymous referees for their valuable comments and suggestions on various aspects of this work.

\renewcommand{\MR}[1]{}
\newcommand{\webref}[1]{\texttt{#1}}

\end{document}